%% file: main.tex
\documentclass{amsart}
\usepackage{amssymb}
\usepackage{amsmath}
\usepackage{graphicx}
\usepackage{manfnt}
\usepackage{amsthm}
\usepackage[mathscr]{euscript}
\usepackage{mathtools}
\usepackage{epstopdf}
\usepackage{bigints}
\usepackage{color}
\usepackage{dsfont}
\usepackage{enumerate}
\usepackage{enumitem}
\usepackage{tikz}
\usepackage{bm}
\usetikzlibrary{arrows}
\usetikzlibrary{positioning}
\usepackage{mathrsfs}
\numberwithin{equation}{section}
\usepackage{hyperref}
\usepackage{tikz-cd}
\usepackage{multicol}
\usepackage{mathrsfs}

\newcommand{\bbm}{\begin{bmatrix}}
\newcommand{\ebm}{\end{bmatrix}}
\newcommand{\bv}{\begin{vmatrix}}
\newcommand{\ev}{\end{vmatrix}}

\newcommand{\g}{\mathfrak{g}}
\newcommand{\lev}{\mathfrak{l}}
\newcommand{\bb}{\mathfrak{b}}
\newcommand{\p}{\mathfrak{p}}

\newcommand{\n}{\mathfrak{n}}
\newcommand{\h}{\mathfrak{h}}

\newcommand{\z}{\mathfrak{z}}

\newcommand{\N}{\mathcal{N}}
\newcommand{\cO}{\mathcal{O}}

\newcommand{\C}{\mathbb{C}}
\newcommand{\Z}{\mathbb{Z}}
\newcommand{\A}{\mathbb{A}}
\newcommand{\beq}{\begin{equation*}}
\newcommand{\eeq}{\end{equation*}}
\newcommand{\beqn}{\begin{eqnarray*}}
\newcommand{\eeqn}{\end{eqnarray*}}

\newcommand{\mf}{\mathfrak}
\newcommand{\mc}{\mathcal}
\newcommand{\bp}{\begin{pmatrix}}
\newcommand{\ep}{\end{pmatrix}}

\DeclareMathOperator{\ch}{ch}
\DeclareMathOperator{\End}{End}
\DeclareMathOperator{\Hom}{Hom}
\DeclareMathOperator{\Ann}{Ann}
\DeclareMathOperator{\Max}{Max}

\DeclareMathOperator{\im}{im }
\DeclareMathOperator{\Int}{Int}

\DeclareMathOperator{\Ext}{Ext}
\DeclareMathOperator{\Vect}{Vect}
\DeclareMathOperator{\ad}{ad}

\theoremstyle{plain}
\newtheorem{theorem}{Theorem}[section]
\newtheorem*{theorem*}{Theorem}
\theoremstyle{plain}
\theoremstyle{definition}
\newtheorem{definition}[theorem]{Definition}
\newtheorem{lemma}[theorem]{Lemma}
\theoremstyle{definition}
\newtheorem{proposition}[theorem]{Proposition}
\theoremstyle{plain}

\theoremstyle{plain}
\newtheorem{corollary}[theorem]{Corollary}
\theoremstyle{remark}
\newtheorem{remark}[theorem]{Remark}
\theoremstyle{definition}

\theoremstyle{definition}

\title{Contravariant pairings between standard Whittaker modules and Verma modules}
\author{Adam Brown and Anna Romanov}
\date{}

\begin{document}

\maketitle

\begin{abstract}
\input{sec-abstract.tex}
\end{abstract}
%\tableofcontents
%\input{sec-outline}
\input{sec-introduction.tex}

\input{sec-algebraic-preliminaries}

\input{sec-geometric-preliminaries}
\input{sec-lie-algebra-homology-of-standard-whittaker-modules}
\input{sec-contravariant-pairings}

\input{sec-costandard-modules}
\input{sec-highest-weight-category.tex}
\input{sec-BGG-Reciprocity}
\input{output-2.bbl}
\end{document}

%% file: sec-abstract.tex
We classify contravariant pairings between standard Whittaker modules and Verma modules over a complex semisimple Lie algebra. These contravariant pairings are useful in extending several classical techniques for category $\mc{O}$ to the Mili\v{c}i\'{c}--Soergel category $\mc{N}$. 
We introduce a class of costandard modules which generalize dual Verma modules, and describe canonical maps from standard to costandard modules in terms of contravariant pairings. 
We show that costandard modules have unique irreducible submodules and share the same composition factors as the corresponding standard Whittaker modules.
We show that costandard modules give an algebraic characterization of the global sections of costandard twisted Harish-Chandra sheaves on the associated flag variety, which are defined using holonomic duality of $\mc{D}$-modules. 
We prove that with these costandard modules, blocks of category $\mc{N}$ have the structure of highest weight categories and we establish a BGG reciprocity theorem for $\N$. 

%% file: sec-introduction.tex
\section{Introduction}
\label{introduction}

In this paper we classify contravariant bilinear pairings between standard Whittaker modules and Verma modules. We use these pairings to adapt several classical proofs for Verma modules and category $\mc{O}$ to the setting of Whittaker modules. We begin with some context.

Let $\mf{g}$ be a semisimple Lie algebra over $\C$.\footnote{ The results in this paper hold more generally for reductive Lie algebras. However, since many of our references assume semisimplicity, we chose to continue working under this assumption.}
% A natural way to construct new representations is to use various finiteness conditions to identify submodules of $V^*$ with desirable properties. For example, if $V$ is a highest weight module relative to a nilpotent subalgebra $\mf{n}$, then the subspace of $\mf{n}$-finite vectors in $V^*$ is again a highest weight module. 
A contravariant form on a $\mf{g}$-module $V$ is a bilinear form on $V$ such that for $v, w \in V$ and $X \in \mf{g}$, $\langle X \cdot v, w \rangle = \langle v, \tau(X) \cdot w \rangle$, where $\tau$ is the transpose antiautomorphism.
If the linear dual $V^* = \Hom_\C(V, \C)$ is given the structure of a $\mf{g}$-module via the action 
\[
X \cdot f(v) = f (\tau(X) \cdot v)
\]
for $X \in \mf{g}$, $f \in V^*$, $v \in V$, then the space of contravariant forms on $V$ is canonically isomorphic to
\[
\Hom_\mf{g}(V,V^*). 
\]
A classical result of Shapovalov shows that if $V$ is a Verma module, this space is 1-dimensional~\cite{Shapovalov}. 
In \cite{BrownRomanov}, the authors generalize Shapovalov's results by classifying contravariant forms on standard Whittaker modules. However, unlike for Verma modules, the space of contravariant forms on a standard Whittaker module is no longer guaranteed to be 1-dimensional. This feature of standard Whittaker modules presents an obstacle in generalizing several constructions for category $\mc{O}$ to Whittaker modules. We provide a brief example of this obstacle below.

In \cite{CompositionSeries}, Mili\v{c}i\'{c}--Soergel introduced a category $\mc{N}$ of $\mf{g}$-modules which interpolates between category $\mc{O}$ and the category of non-degenerate Whittaker modules (Definition \ref{def:category-N}). Category $\mc{N}$ contains category $\mc{O}$ as a full subcategory, and can be viewed as a natural generalization of it. Standard Whittaker modules, which are cyclically generated by a vector on which the nilpotent radical $\mf{n}$ of a Borel subalgebra of $\mf{g}$ acts by a character $\eta\in\mf{n}^\ast$, play the role of Verma modules in category $\N$. One of the celebrated features of category $\mc{O}$ is that its blocks have the structure of highest weight categories~\cite{CPS}, with costandard modules defined as the $\mf{n}$-finite vectors in the linear dual of a Verma module. It is natural to ask if this structure extends to $\N$. While it is known that blocks of category $\N$ with regular integral infinitesimal character are highest weight categories (this follows by reducing the problem to a singular block of category $\mc{O}$ \cite{CompositionSeries}), 
the results of \cite{BrownRomanov} show that the straight-forward generalization of the above definition of costandard modules does not yield the same result for $\N$: defining costandard modules to be spaces of $\eta$-twisted $\mf{n}$-finite vectors in the linear dual of a standard Whittaker module does not give blocks of $\mc{N}$ the structure of highest weight categories.  The problem is the surplus of contravariant forms.
 
In this paper, we address these disparities by giving an alternate generalization of Shapovalov's results. Our main contribution is the classification of \emph{contravariant pairings} (Definition \ref{contravariant pairing def}) between standard Whittaker modules and Verma modules.
\begin{theorem*}[Theorem \ref{thm:unique-forms}]
Assume $\lambda+\rho\in \h^*$ is regular. Let $M(\lambda,\eta)$ be a standard Whittaker module (Definition \ref{standard Whittaker module}) and $M(\mu)$ be a Verma module with highest weight $\mu\in\h^*$. Let $W_\eta$ be the subgroup of the Weyl group of $\mf{g}$ determined by $\eta\in\mf{n}^*$ (Section \ref{sec:standard-whittaker-modules}). Then 
\[
\Hom_\mf{g}(M(\lambda,\eta),M(\mu)^\ast)=\begin{cases}
\mathbb{C}&\text{ if $\mu\in W_\eta\cdot\lambda$}\\
0&\text{ else.}
\end{cases}
\]
\end{theorem*}

Unlike contravariant forms on standard Whittaker modules, these contravariant pairings are unique up to scaling, a feature which more closely resembles the case of Verma modules. With this generalization, we extend well-known arguments for Verma modules directly to standard Whittaker modules. We give 5 examples. 
\begin{enumerate}
    \item We define a costandard module in $\N$ to be the space of $\eta$-twisted $\mf{n}$-finite vectors in the linear dual of a Verma module (Definition \ref{def:costandard}). With this definition, contravariant pairings between standard Whittaker modules and Verma modules induce canonical maps between standard and costandard modules in $\N$ (Lemma \ref{lemma:costandard-objects}). 
    \item We show that costandard modules have unique irreducible submodules and share the same composition factors as the corresponding standard Whittaker modules, and that these properties uniquely define the costandard modules up to isomorphism (Theorem \ref{thm:universal-properties}).
    \item We prove that costandard modules align under Beilinson--Bernstein localization with costandard $\eta$-twisted Harish-Chandra sheaves on the associated flag variety (Lemma \ref{lemma:algebraic-geometric-agreement}). 
    \item We prove that our definitions give blocks of category $\N$ the structure of highest weight categories (Corollary \ref{cor:highest-weight}). 
    \item We prove a Beilinson-Gelfand-Gelfand reciprocity theorem for category $\N$ (Theorem \ref{thm:bgg}). 
\end{enumerate}
Using contravariant pairings, we are able to generalize the proofs of these classical results for category $\mc{O}$ in a way that clearly follows the structure of the original arguments. In each of the above cases, when we set $\eta=0$, we recover a traditional proof for category $\mc{O}$.

\begin{remark}
In \cite{CompositionSeries}, certain equivalences between blocks of category $\mc{N}$ with regular integral infinitesimal characters and singular blocks of category $\mc{O}$ are established. One could alternatively define costandard modules as those corresponding to dual Verma modules under these equivalences. Using this approach, results analogous to (2), (3), (4), and (5) above could be deduced directly from the corresponding results in category $\mc{O}$. Our main results about contravariant pairings described above do not follow from these equivalences. 
\end{remark}

This paper is organized as follows. In Section \ref{algebraic preliminaries}, we define category $\mc{N}$ and establish algebraic background, including a classification of standard and simple modules, the construction of Whittaker functors \cite{Backelin}, and a review of Lie algebra (co)homology. In Section \ref{geometric preliminaries}, we define the category of twisted Harish-Chandra sheaves and establish geometric preliminaries, including a classification simple objects, some necessary results about Beilinson--Bernstein localization, and a method for computing Lie algebra homology geometrically. In Section \ref{lie algebra homology of standard whittaker modules}, we compute the Lie algebra homology of standard Whittaker modules (Theorem \ref{thm:homology}) both algebraically and geometrically, which provides our main tool when classifying contravariant pairings in Section \ref{contravariant pairings of standard whittaker modules}. In Section \ref{contravariant pairings of standard whittaker modules}, we define and classify contravariant pairings between standard Whittaker modules and Verma modules (Theorem \ref{thm:unique-forms}). We then give an explicit construction of these contravariant pairings (Theorem \ref{thm:shapovalov}). In Section \ref{costandard whittaker modules}, we define costandard modules in category $\mc{N}$ (Definition \ref{def:costandard}). We give a set of universal properties for costandard modules (Theorem \ref{thm:universal-properties}), and show that contravariant pairings induce morphisms from standard modules to costandard modules in $\mc{N}$. In Section \ref{whittaker modules form a highest weight category}, we give a geometric proof that blocks of $\mc{N}$ are highest weight categories (Theorem \ref{thm:geometric-highest-weight}, Corollary \ref{cor:highest-weight}). We conclude with Section \ref{sec:BGG-Reciprocity}, which uses our results to prove a Bernstein--Gelfand--Gelfand reciprocity formula for $\N$ (Theorem \ref{thm:bgg}).

\section*{Acknowledgements}
We thank Catharina Stroppel and Jens Niklas Eberhardt for interesting discussions. The first author acknowledges the support of the European Union’s Horizon 2020 research and innovation programme under the Marie Skłodowska-Curie Grant Agreement No.\ 754411. The second author is supported by the National Science Foundation Award No.\ 1803059 and the Australian Research Council grant DP170101579.

\section*{Index of Notation}
\label{index of notation}
\setlength{\columnsep}{30pt}
\begin{multicols}{2}

 \begin{itemize}[ leftmargin=*]
 \setlength\itemsep{.75em}
     \item $\g$ semisimple Lie algebra over $\C$
     \item $\mf{h} \subset \mf{b} \subset \g$ fixed Cartan subalgebra and Borel subalgebra in $\g$
     \item $\Sigma \subset \h^*$ root system of $\g$
     \item $\Pi \subset \Sigma^+ \subset \Sigma$ simple and positive roots in $\Sigma$ determined by $\mf{b}$ 
     \item $(W,S)$ associated Coxeter system
     \item $\rho = \frac{1}{2} \sum_{\alpha \in \Sigma^+} \alpha$
     \item $w \cdot \lambda := w(\lambda+ \rho) - \rho$ for $w \in W$, $\lambda \in \h^*$, the dot action 
     \item $\g_\alpha = \{ x \in \g \mid [h,x]=\alpha(h)x\}$ root space for $\alpha \in \Sigma$
     \item $\{y_\alpha, x_\alpha\}_{\alpha \in \Sigma^+} \cup \{h_\alpha\}_{\alpha \in \Pi}$ Chevalley basis of $\g$
     \item $\mf{n}=[\mf{b},\mf{b}]= \bigoplus_{\alpha \in \Sigma^+} \g_\alpha$ nilpotent radical of $\mf{b}$
     \item $\bar{\mf{n}}=\bigoplus_{\alpha \in \Sigma^+}\g_{-\alpha}$
     \item $U(\mf{a})$ universal enveloping algebra of Lie algebra $\mf{a}$
     \item $Z(\mf{a})$ center of $U(\mf{a})$
     \item $p:Z(\g)\rightarrow U(\h)$ Harish-Chandra homomorphism \cite[Ch. 1 \S7]{BGGcatO}
     \item $\chi^\lambda: Z(\g)\rightarrow \C; z \mapsto \lambda \circ p(z)$ infinitesimal character for $\lambda \in \h^*$
     \item $\xi:\h^\ast\rightarrow \Max Z(\g)$, $\xi(\lambda)=\ker\chi^\lambda$
     \item $\ch \mf{n}:= \{\eta: \n \rightarrow \C  \mid \eta \text{ is a }\\ \text{Lie algebra homomorphism}\}$
     \item $\Pi_\eta:=\{\alpha \in \Pi \mid \eta|_{\g_\alpha} \neq 0 \}$ set of simple roots determined by $\eta \in \ch \n$
     \item $\Pi_\eta\subset \Sigma_\eta^+\subset \Sigma_\eta \subset \h^*$ root system generated by $\Pi_\eta$
     \item $W_\eta \subset W$ Weyl group of $\Pi_\eta$
     \item $\mf{l}_\eta = \h \oplus \bigoplus_{\alpha \in \Sigma_\eta}\g_\alpha$ Levi subalgebra of $\g$ determined by $\eta \in \ch \n$
     \item $\mf{n}_\eta = \bigoplus_{\alpha \in \Sigma_\eta^+}\g_\alpha$ nilradical of $\mf{l}_\eta$
     \item $\bar{\n}_\eta = \bigoplus_{\alpha \in \Sigma_\eta^-}\g_\alpha$
     \item $\mf{n}^\eta = \bigoplus_{\alpha \in \Sigma^+-\Sigma_\eta^+} \g_\alpha$
    \item $\mf{p}_\eta = \mf{l}_\eta \oplus \n^\eta$ parabolic subalgebra containing $\mf{l}_\eta$
    \item $p_\eta: Z(\mf{l}_\eta)\rightarrow U(\h)$ Harish-Chandra homomorphism for $Z(\mf{l}_\eta)$
    \item $\chi_\eta^\lambda: Z(\lev_\eta)\rightarrow \C; z \mapsto \lambda \circ p_\eta(z)$ infinitesimal character for $\lambda \in \h^*$
    \item $\xi_\eta:\h^\ast\rightarrow \Max Z(\lev_\eta)$, $\xi_\eta(\lambda)=\ker\chi_\eta^\lambda$
 \end{itemize}
 \end{multicols}

%% file: sec-algebraic-preliminaries.tex
\section{Algebraic Preliminaries}
\label{algebraic preliminaries}
Our algebraic setting is a category $\mc{N}$ of $\g$-modules introduced by Mili\v{c}i\'{c}--Soergel in \cite{CompositionSeries}. Each simple module in $\mc{N}$ is a Whittaker module, and $\mc{N}$ contains all highest weight modules. In particular, $\mc{N}$ contains Bernstein--Gelfand--Gelfand's category $\mc{O}$ \cite{BGG} as a full subcategory. In this section, we define the Mili\v{c}i\'{c}--Soergel category $\mc{N}$ and list some of its basic properties. 
\subsection{A category of Whittaker modules}
\label{a category of whittaker modules}

Fix a Cartan subalgebra $\mf{h}$ of $\mf{g}$ contained in a Borel subalgebra $\mf{b}$, and let $\g = \bar{\n} \oplus \h \oplus \n$ be the corresponding triangular decomposition, with $\n = [\mf{b}, \mf{b}]$. We denote by $\ch{\mf{n}}$ the set of Lie algebra homomorphisms $\eta:\mf{n} \rightarrow \C$. For any Lie algebra $\mf{a}$, we denote by $U(\mf{a})$ its universal enveloping algebra and $Z(\mf{a})$ the center of $U(\mf{a})$. Denote by $\Max{Z(\mf{a})}$ the set of maximal ideals in $Z(\mf{a})$. Any character $\eta \in \ch{\mf{n}}$ can be extended to an algebra homomorphism $\eta:U(\mf{n})\rightarrow \C$ which we will call by the same name. We denote by $\ker \eta \subset U(\mf{n})$ the kernel of the algebra homomorphism.

\begin{definition}
A {\em Whittaker vector of type $\eta \in \ch{\n}$} in a $U(\mf{g})$-module $V$ is a vector $w \in V$ such that $u\cdot w = \eta(u)w$ for all $u \in U(\n)$. An {\em $\eta$-Whittaker module} is a $U(\mf{g})$-module which is cyclically generated by a Whittaker vector of type $\eta$. 
\end{definition}

In \cite[ \S1]{CompositionSeries}, Mili\v{c}i\'{c}--Soergel introduced a category $\mc{N}$ whose simple objects are  irreducible Whittaker modules. 

\begin{definition} 
\label{def:category-N}
Let $\N$ be the category of $U(\g)$-modules which are
\begin{enumerate}
    \item finitely generated,
    \item locally $U(\n)$-finite, and
    \item locally $Z(\g)$-finite. 
\end{enumerate}
\end{definition}
\begin{proposition}\cite[ Lem. 2.1, Lem. 2.2]{TwistedSheaves}
\label{prop:Ndefinition}
The category $\mc{N}$ decomposes into
\[
\N  = \bigoplus_{I\in \text{Max}Z(\g)}\bigoplus_{\eta\in\ch\n} \N\left(\widehat{I},\eta\right),
\]
where $\N\left(\widehat{I},\eta\right)$ is the full subcategory of $\mc{N}$ consisting of objects $M \in \mc{N}$ satisfying the following two conditions:
\begin{enumerate}[label = (\roman*)]
    \item $M$ is annihilated by a power of $I\in\Max Z(\g)$; 
    \item $M$ is locally annihilated by a power of $\ker \eta$.
    %for all $m \in M$ and $u \in U(\n)$, there exists $k \in \Z_{\geq 0}$ such that $(u - \eta(u))^km=0$. 
\end{enumerate}%which are annihilated by a power of $\chi$ and are generalized weight spaces for the character $\eta$: for all $m\in M$ and $u\in\n$, there exists $k\in\Z_{\ge0} $ such that $(u-\eta(u))^km=0$. 
\end{proposition}

Let $\mc{N}(I, \eta)$ be the subcategory of $\mc{N}\left(\widehat{I}, \eta\right)$ consisting of modules annihilated by $I$. Each irreducible Whittaker module lies in some $\mc{N}(I, \eta)$. 

\subsection{Standard and simple Whittaker modules}
\label{sec:standard-whittaker-modules}
Let $\Pi \subset \Sigma^+ \subset \Sigma\subset \h^*$ be the simple and positive roots in the root system of $\g$ determined by our choice of $\mf{b}$, and let $(W,S)$ be the associated Coxeter system. For a root $\alpha \in \Sigma$, we denote by $\g_\alpha = \{ x \in \g \mid [h,x]=\alpha(h)x\}$ the corresponding root space. With this notation, we have $\mf{n}= \bigoplus_{\alpha \in \Sigma^+} \g_\alpha$.

A character $\eta \in \ch{\mf{n}}$ determines a subset of simple roots:
\[
\Pi_\eta:=\{\alpha \in \Pi \mid \eta|_{\g_\alpha} \neq 0 \}.
\]
Let $\Sigma_\eta \subset \h^*$ be the root system generated by $\Pi_\eta$ and $W_\eta$ the corresponding Weyl group. From $\eta$ we obtain several Lie subalgebras of $\g$. In particular, we name
\[
\mf{l}_\eta = \h \oplus \bigoplus_{\alpha \in \Sigma_\eta}\g_\alpha, \hspace{2mm} \mf{n}_\eta = \bigoplus_{\alpha \in \Sigma_\eta^+}\g_\alpha, \hspace{2mm}\bar{\n}_\eta = \bigoplus_{\alpha \in \Sigma_\eta^-}\g_\alpha, \hspace{2mm} \mf{n}^\eta = \bigoplus_{\alpha \in \Sigma^+-\Sigma_\eta^+} \g_\alpha, \hspace{2mm} \mf{p}_\eta = \mf{l}_\eta \oplus \n^\eta.  
\]

Let $p_\eta:Z(\mf{l}_\eta)\rightarrow U(\h)$ be the Harish-Chandra homomorphism of $U(\mf{l}_\eta)$. For each $\lambda \in \h^\ast$, denote by $\chi^\lambda_\eta: Z(\mf{l}_\eta)\rightarrow \C$, $z \mapsto (\lambda \circ p_\eta)(z)$ the corresponding infinitesimal character. We have $\chi^\lambda_\eta = \chi_\eta^\mu$ if and only if $\mu \in W_\eta \cdot \lambda$. Let $\xi_\eta:\h^* \rightarrow \Max{Z(\mf{l}_\eta)}$, $\lambda \mapsto \ker{\chi^\lambda_\eta}$ the map associating elements of $\h^*$ to maximal ideals in $Z(\mf{l}_\eta)$. 

From the data $(\lambda, \eta) \in \h^* \times \ch{\mf{n}}$, we construct a $U(\mathfrak{l}_\eta)$-module
\begin{equation}
    \label{nondegenerate whittaker}
    Y(\lambda, \eta):=U(\mf{l}_\eta)  \otimes_{Z(\mf{l}_\eta) \otimes U(\mf{n}_\eta)} \mathbb{C}_{\chi^\lambda_\eta, \eta}.
\end{equation}

Here $\C_{\chi^\lambda_\eta, \eta}$ is the one-dimensional $Z(\mf{l}_\eta) \otimes U(\mf{n}_\eta)$-module with action 
\[u \otimes x \cdot z = \chi^\lambda_\eta(u)\eta(x)z\]
for $u \in Z(\mf{l}_\eta)$, $x \in U(\mf{n}_\eta)$, $z \in \C$. By construction, we have $Y(\lambda, \eta) \cong Y(\mu, \eta)$ if and only if $\mu \in W_\eta \cdot \lambda$. Here $\cdot$ denotes the ``dot action'': for $w \in W$ and $\lambda \in \h^*$, $w \cdot \lambda = w(\lambda + \rho) - \rho$, where $\rho =\frac{1}{2} \sum_{\alpha \in \Sigma^+} \alpha$. For any $\lambda \in \h^*$, $Y(\lambda, \eta)$ is an irreducible nondegenerate $\eta$-Whittaker module for $U(\mf{l}_\eta)$ \cite[\S2 Prop.  2.3]{McDowell}. 

Standard objects in the category $\mc{N}$ are constructed by parabolically inducing the irreducible $U(\mf{l}_\eta)$-modules $Y(\lambda, \eta)$.

\begin{definition} 
\label{standard Whittaker module}
For $(\lambda, \eta) \in \h^* \times \ch{n}$, define the $U(\g)$-module
\[
M(\lambda, \eta) := U(\g) \otimes_{U(\mf{p}_\eta)}Y(\lambda, \eta).
\]
Here $Y(\lambda, \eta)$ (\ref{nondegenerate whittaker}) is viewed as a $\mathcal{U}(\mf{p}_\eta)$-module via the natural morphism $\mf{p}_\eta \rightarrow \mf{l}_\eta$. We call the modules $M(\lambda, \eta)$ {\em standard Whittaker modules}.
\end{definition}

\begin{proposition}\cite[Prop. 2.4, Thm. 2.5, Thm. 2.9]{McDowell}
\label{whittaker facts}
The standard Whittaker module $M(\lambda, \eta)$ satisfies the following properties.
\begin{enumerate}[label=(\roman*)]
    \item Let $\xi: \h^* \rightarrow \Max Z(\g)$ be the map associating $\lambda \in \h^*$ to the kernel of the corresponding infinitesimal character $\chi^\lambda$. Then $M(\lambda, \eta) \in \mc{N}(\xi(\lambda), \eta)$.
    \item  Two modules $M(\lambda, \eta)$ and  $M(\mu, \eta)$ are isomorphic if and only if $\mu \in W_\eta \cdot \lambda$.
    \item The module $M(\lambda, \eta)$ is an $\eta$-Whittaker module generated by the Whittaker vector $\omega =1 \otimes 1 \otimes 1$.
    \item The center $\mf{z}$ of the reductive Lie algebra $\mf{l}_\eta$ is $\mf{z}=\{h \in \mf{h} \mid \alpha(h) = 0, \alpha \in \Pi_\eta\}$. For $\lambda \in \h^*$, we denote by $\lambda_\mf{z}$ the restriction of $\lambda$ to $\mf{z}$. There is a natural partial order on $\mf{z}^*$ obtained from the partial order on $\mf{h}^*$. The Lie algebra $\mf{z}$ acts semisimply on $M(\lambda, \eta)$, and $M(\lambda, \eta)$ decomposes into $\mf{z}$-weight spaces
    \[
    M(\lambda, \eta)=\bigoplus_{\nu_\mf{z} \leq \lambda_\mf{z}} M(\lambda, \eta)_{\nu_\mf{z}}.
    \]
    Each $\mf{z}$-weight space $M(\lambda, \eta)_{\nu_\mf{z}}$ is a $U(\mf{l}_\eta)$-module. Furthermore, as $U(\mf{l}_\eta)$-modules,  $M(\lambda, \eta)_{\lambda_\mf{z} }\cong Y(\lambda, \eta)$ and $M(\lambda, \eta)_{\nu_\mf{z}}\cong U(\bar{\mf{n}}^\eta)_{\mu_\mf{z}}\otimes_\mathbb{C}Y(\lambda, \eta)$, where  $\nu_\mf{z} = \mu_\mf{z} + \lambda_{\mf{z}}$ and $\mu_\mf{z}\leq0$ is a $\mf{z}$-weight of $U(\bar{\mf{n}}^\eta).$
    \item $M(\lambda, \eta)$ has a unique irreducible quotient, denoted $L(\lambda, \eta)$. All irreducible objects in $\mc{N}$ appear as such quotients. 
\end{enumerate}
\end{proposition}
\begin{remark}
If $\lambda+\rho$ is regular, there is a unique element $\mu \in W_\eta \cdot \lambda$ such that $\mu+\rho$ is dominant with respect to $\Sigma_\eta^+$; that is, $\alpha^\vee(\mu+\rho) \geq 0$ for all $\alpha \in \Sigma_\eta^+$. For the remainder of the paper, unless otherwise stated, we assume that $\mu$ is chosen to be this unique dominant element when we write $M(\mu, \eta)$. (Proposition \ref{whittaker facts}(ii) guarantees that such a choice can be made.)  

 %Because $M(\lambda,\eta)= M(\mu,\eta)$ if and only if $\mu\in W_\eta\cdot\lambda$, there is some ambiguity in which weight we use to parametrize standard Whittaker modules. It will often be convenient to fix an element of the $W_\eta$ orbit of $\lambda$ in order to avoid these ambiguities. Because we are working in the setting of regular weights, there is a unique element in $W_\eta\cdot\lambda$ which is dominant relative to the partial order on weights induced by the positive roots $\Sigma^+_\eta$. In other words, there is a unique element in $W_\eta\cdot\lambda$ which is dominant relative to the fixed positive roots system of the subalgebra $\lev_\eta$. Therefore, unless otherwise stated, when we write $M(\lambda,\eta)$, we assume $\lambda$ is chosen to be dominant relative to $\Sigma^+_\eta$. 
\end{remark}
\begin{definition}
\label{nondegenerate}
We say that a character $\eta\in \ch{\mf{n}}$ is {\em nondegenerate} if $\Pi_\eta = \Pi$. We say a Whittaker module is {\em nondegenerate} if it is an $\eta$-Whittaker module for a nondegenerate character $\eta$. 
\end{definition}

\begin{remark}
 If $\eta = 0$, then $\mf{l}_\eta = \mf{h}$ and $M(\lambda, \eta)$ is the Verma module of highest weight $\lambda$ (which we denote by $M(\lambda)$). If $\eta$ is nondegenerate, then $\mf{l}_\eta = \mf{g}$ and $M(\lambda, \eta) = Y(\lambda, \eta)$ is irreducible.  
\end{remark}
\subsection{Whittaker functors}\label{sec:Whittaker-functors}

Given a $U(\g)$-module $X$, let $(X)_\eta$ denote the space of $\eta$-twisted $U(\n)$-finite vectors in $X$:
\begin{equation}
    \label{eta twisted n-finite vectors}
    (X)_\eta := \{ x\in X:\text{$\forall u\in \n$, $\exists$ $k$ s.t. }(u-\eta(u))^kx=0\}. 
\end{equation}

For a $U(\g)$-module $X$ in category $\mc{O}$\footnote{Category $\mc{O}$ is the category of $U(\g)$-modules which are finitely generated, $\h$-semisimple, and locally $U(\n)$-finite. For $\lambda \in \mf{h}^*$, denote by $\mc{O}_\lambda$ the subcategory of $\mc{O}$ consisting of modules whose composition factors are isomorphic to $L(w \cdot \lambda)$ for $w \in W_\lambda$. Here  for $\mu \in \mf{h}^*$, $L(\mu)$ denotes the unique irreducible quotient of $M(\mu)$ and $W_\lambda$ is the integral Weyl group of $\lambda$.}, denote by $\overline{X}$ the formal completion; i.e. if $X=\bigoplus_{\lambda \in \h^*} X_\lambda$, then $\overline{X} = \prod_{\lambda \in \h^*} X_\lambda$. Denote by $\overline{\Gamma}_\eta(X):= (\overline{X})_\eta$. In~\cite[\S3]{Backelin} it is shown that $\overline{\Gamma}_\eta$ defines an exact functor 
\begin{align*}
\overline{\Gamma}_\eta: \mc{O}_\lambda \rightarrow \mc{N}(\xi(\lambda),\eta)
\end{align*}
 for any $\lambda\in\h^\ast$. We refer to $\overline{\Gamma}_\eta$ as a {\em Whittaker functor}.

\begin{proposition}[{\cite[Prop. 6.9]{Backelin}}]\label{prop:Whittaker-functor}
Let $\lambda\in\h^\ast$. For each $w\in W_\eta$ 
    \[
\overline{\Gamma}_\eta (M(w\cdot\lambda))= M(\lambda,\eta). 
\]

\end{proposition}

\subsection{Twisted and untwisted Lie algebra (co)homology}

Our arguments in upcoming sections will make use of (twisted) Lie algebra (co)homology. 

\begin{definition}
\label{Lie algebra homology}
Let $X$ be a left $U(\g)$-module. The {\em $\bar{\mf{n}}$-homology} of $X$ is 
\[
H_k(\bar{\mf{n}},X) := \text{Tor}_k^{U(\bar{\mf{n}})}(\C, X),
\]
where $\C$ is the trivial right $U(\bar{\mf{n}})$-module. We are primarily interested in the degree zero $\bar{\mf{n}}$-homology: 
\begin{align*}
H_0(\bar{\mf{n}}, X) &= \C \otimes_{U(\bar{\mf{n}})} X \\
&= X/\bar{\mf{n}}X 
\end{align*}
We refer to $H_0(\bar{\mf{n}},X)$ as the \emph{$\bar{\mf{n}}$-coinvariants of $X$}.
\end{definition}

The vector space $X/\bar{\mf{n}}X$ has a natural structure of an $\h$-module, so degree zero $\bar{\mf{n}}$-homology defines a right exact covariant functor 
\[
H_0(\bar{\mf{n}}, -): U(\g)\text{-mod} \rightarrow U(\h) \text{-mod}.
\]

\begin{definition}
\label{twisted Lie algebra cohomology}
Let $X$ be a left $U(\n)$-module and $\eta\in\ch \n$. The $\eta$-twisted $\n$-cohomology of $X$ is defined to be 
\[
H_\eta^k(\n,X):=\text{Ext}^k_{U(\n)}(\C_\eta, X). 
\]
\end{definition}
The $\eta$-twisted $\n$-cohomology of $X$ in degree 0 is the subspace of Whittaker vectors in $X$:
\begin{align*}
H^0_\eta(\n,X)& = \Hom_{U(\n)}(\mathbb{C}_\eta,X)\\
&=\{x\in X:(u-\eta(u))x=0\quad\forall u\in U(\n)\}.
\end{align*}
When $\eta=0$, we refer to $\eta$-twisted $\n$-cohomology as just $\n$-cohomology and drop $\eta$ from our notation:
\[
H^k(\mf{n}, X) := H^k_0(\mf{n}, X) = \Ext_{U(\mf{n})}^k(\C, X).
\]
Here $\C$ is the trivial representation of $U(\n)$.

%\begin{definition}
%\label{contravariant form}
%Let $\tau:U(\g)\rightarrow U(\g)$ be the transpose antiautomorphism defined by $\tau(x_\alpha)=y_{\alpha}$ and $\tau(h_\alpha)=h_\alpha$. A \emph{contravariant form} on a $U(\g)$-module $X$ is a symmetric bilinear form $\langle\cdot , \cdot \rangle:X \times X \rightarrow \C$ such that 
%$$
%\langle uv,w\rangle =\langle v,\tau(u)w\rangle
%$$
%for all $u\in U(\g)$ and $v,w\in X$.
%\end{definition}
%\begin{theorem}[\cite{BrownRomanov}]
%When $\eta$ is a nondegenerate character of $\n$, the space of contravariant forms on $M(\lambda,\eta)$ has dimension equal to the cardinality of $W_\eta$. 
%\end{theorem} 

%% file: sec-geometric-preliminaries.tex
\section{Geometric Preliminaries}
\label{geometric preliminaries}

Our geometric setting is a category of twisted equivariant $\mc{D}$-modules on the flag variety of $\g$, which we refer to as {\em twisted Harish-Chandra sheaves}. In \cite[ \S1]{TwistedSheaves}, Mili\v{c}i\'{c}--Soergel establish that these twisted sheaves correspond to blocks of category $\mc{N}$ under Beilinson--Bernstein localization. In this section, we introduce this geometric category and list some basic properties of twisted Harish-Chandra sheaves. 

\subsection{Twisted Harish-Chandra sheaves} Fix $\lambda \in \mf{h}^*$ and $\eta \in \ch \n$. Let $X$ be the flag variety of $\mf{g}$ and $\mc{D}_\lambda$ the $\lambda$-twisted sheaf of differential operators on $X$ \cite[Ch. 1, \S 2]{Localization}. Let $\ell:W \rightarrow \mathbb{N}$ be the length function on the Weyl group of $\mf{g}$. The action of the group $N = \Int{\mf{n}}$ stratifies $X$ into Bruhat cells $C(w)\cong \A^{\ell (w)}$ parameterized by elements $w \in W$. 

\begin{definition}\label{def:HCsheaf}  (\cite[\S1]{TwistedSheaves}. See also \cite[App. B]{HMSWI} and \cite[\S4]{MP}.)
An {\em $\eta$-twisted Harish-Chandra sheaf} for the Harish-Chandra pair $(\mf{g},N)$ is a coherent $\mc{D}_\lambda$-module $\mc{V}$  satisfying the following conditions: 
\begin{enumerate}[label=(\roman*)]
    \item $\mc{V}$ is $N$-equivariant as an $\mc{O}_X$-module \cite[Ch.1 \S 3 Def. 1.6]{Mumford};
    \item the action morphism $\mc{D}_\lambda \otimes_{\mc{O}_X} \mc{V} \rightarrow \mc{V}$ is a morphism of $N$-equivariant $\mc{O}_X$-modules;
    \item the actions of $\mc{D}_\lambda$ and $N$ differ by $\eta$; i.e. for all $x \in \mf{n}$
    \[
    \pi(x) - \mu(x) = \eta(x),
    \]
    where $\pi$ is the action on $\mc{V}$ induced by the map $\mf{n} \rightarrow U(\mf{g}) \rightarrow \mc{D}_\lambda$ and $\mu$ is the differential of the $N$-action. 
\end{enumerate}
\end{definition}
Note that condition (i) involves extra data on the $\mc{D}_\lambda$-module and conditions (ii) and (iii) are assumptions. A morphism of $\eta$-twisted Harish-Chandra sheaves is a $\mc{D}_\lambda$-module morphism which is also a morphism of $N$-equivariant $\mc{O}_X$-modules. We denote by $\mc{M}_{\text{coh}}(\mc{D}_\lambda, N, \eta)$ the category of $\eta$-twisted Harish-Chandra sheaves. Because $N$ is connected, any $\mc{O}_X$-module can only have one possible $N$-equivariant structure, so we can consider the category $\mc{M}_{\text{coh}}(\mc{D}_\lambda, N, \eta)$ as a full subcategory of the category of $\mc{D}_\lambda$-modules.

\subsection{Standard and simple twisted Harish-Chandra sheaves}
\label{sec:standard-and-simple-HC-sheaves}
Within the category $\mc{M}_{\text{coh}}(\mc{D}_\lambda, N, \eta)$ there are standard, costandard, and simple objects parameterized by $W_\eta \backslash W$. They are constructed as follows. (See \cite[\S3]{TwistedSheaves} and \cite[\S3.1]{Romanov} for more details.) For a coset $C\in W_\eta\backslash W$, let $w^C$ be the unique longest element of $C$ \cite[Ch.6, \S1]{Localization}, and $i_{w^C}:C(w^C)\rightarrow X$ be the inclusion of the Bruhat cell $C(w^C)$ into the flag variety. There is a unique irreducible connection on $C(w^C)$ satisfying the compatibility condition (iii) in Definition \ref{def:HCsheaf}; we denote it by $\mc{O}_{C(w^C), \eta}$.  As an $N$-equivariant $\mc{O}_X$-module, $\mc{O}_{C(w^C), \eta}$ is isomorphic to $\mc{O}_X$, but the $\mc{D}_\lambda$-module structure is twisted by $\eta$. 
\begin{definition}
\label{def:standard-sheaf}
For each coset $C \in W_\eta \backslash W$, we define the corresponding {\em standard\footnote{Note that the terminology here differs from \cite{TwistedSheaves, Romanov}, where the $!$-pushforward is called costandard and the $+$-pushforward standard. We chose the opposite terminology in this paper so that the global sections of (co)standard $\eta$-twisted Harish-Chandra sheaves are (co)standard Whittaker modules, which seems to us more natural.}  $\eta$-twisted Harish-Chandra sheaf}
\[
\mc{M}(w^C,\lambda,\eta) := i_{w^C!}(\cO_{C(w^C),\eta}),
\]
and {\em costandard $\eta$-twisted Harish-Chandra sheaf}
\[
\mc{I}(w^C,\lambda,\eta) := i_{w^C+}(\cO_{C(w^C),\eta}).
\]
The $!$-pushforward functor $i_{w^C!}$ is defined by pre- and post- composing $i_{w^C+}$ with holonomic duality (see \cite[Def. 5]{Romanov} and more generally \cite[App. A.2]{Romanov} for conventions with $\mc{D}$-module functors).
\end{definition}

The sheaf $\mc{M}(w^C, \lambda, \eta)$ has a unique irreducible quotient, which we denote by $\mc{L}(w^C, \lambda, \eta)$ \cite[Prop. 3]{Romanov}. The sheaf $\mc{L}(w^C, \lambda, \eta)$ is isomorphic to the unique irreducible submodule of $\mc{I}(w^C, \lambda, \eta)$. All simple objects in $\mc{M}_{\text{coh}}(\mc{D}_\lambda, N, \eta)$ occur in this way \cite[\S3]{TwistedSheaves}. 
\begin{remark}
 The $\eta$-twisted connection $\mc{O}_{C(w^C), \eta}$ can also be described in terms of exponential $\mc{D}$-modules. Let $\mathbb{G}_a \cong \mathbb{A}^1$ be the additive group. The {\em exponential $\mc{D}$-module} on $\mathbb{G}_a$ is
\[
\mathrm{exp}:=D_{\mathbb{G}_a}/D_{\mathbb{G}_a}(\partial - 1),
\]
where $D_{\mathbb{G}_a}$ denotes global differential operators on $\mathbb{A}^1$ (the Weyl algebra), generated by $\partial$ and $z$. Corresponding to a Lie algebra character $\eta: \mf{n} \rightarrow \C$ is a group character $\eta:N \rightarrow \mathbb{G}_a$ which we call by the same name. For certain Bruhat cells, we can use $\eta$ to construct an exponential $\mc{D}$-module on the Bruhat cell. In particular, if $C(w)$ has the property that $\eta|_{\mathrm{stab}_N x} = 1$ for all $x \in C(w)$, then $\eta$ factors through the quotient $N/\mathrm{stab}_N x \cong C(w)$, 
\[
\begin{tikzcd}
N \arrow[r] \arrow[rr, bend left, "\eta"] & N/ \mathrm{stab}_N x \arrow[r, "\bar{\eta}"] & \mathbb{G}_a
\end{tikzcd}
\]
so we can define a $\mc{D}$-module $\bar{\eta}^! \mathrm{exp}$ on $C(w)$. It turns out that $\eta|_{\mathrm{stab}_n x} = 1$ for $x \in C(w)$ if and only if $w=w^C$ is the longest coset representative for some coset $C \in W_\eta \backslash W$ (see proof of Lemma 4.1 in \cite{TwistedSheaves}). The $\mc{D}_{C(w)}$-modules constructed in this way are exactly the $\eta$-twisted connections $\mc{O}_{C(w), \eta}$ for $\lambda = - \rho$. 
\end{remark}

The relationship between standard, costandard, and simple twisted Harish-Chandra sheaves can be described in terms of six functor formalism on derived categories of $\mc{D}$-modules. Let $D:=D^b(\mc{M}_{qc}(\mc{D}_\lambda))$ be the bounded derived category of quasi-coherent $\mc{D}_\lambda$-modules on $X$, and $D_{w^C}:=D^b(\mc{M}_{qc}(\mc{D}_\lambda^{i_{w^C}}))$ be the bounded derived category of quasi-coherent $\mc{D}_\lambda^{i_{w^C}}$-modules on $C(w^C)$. (Here $\mc{D}_\lambda^{i_{w^C}}$ is the twisted sheaf of differential operators on $C(w^C)$ obtained by pulling back $\mc{D}_\lambda$ via $i_{w^C}$, see \cite[Ch. 1 \S1]{Localization}. These are $\lambda$-twisted differential operators on the Bruhat cell $C(w)$.) The $!$-pushforward functor 
\[
i_{w^C!}:D_{w^C} \rightarrow D
\]
is a left adjoint to the restriction functor 
\[
i_{w^C}^\bullet:D \rightarrow D_{w^C}.
\]
Hence  
\begin{align*}
\Hom_D(\mc{M}(w^C, \lambda, \eta), \mc{I}(w^C, \lambda, \eta)) &=\Hom_{D}(i_{w^C!}\mc{O}_{C(w^C), \eta}, i_{w^C+}\mc{O}_{C(w^C), \eta}) \\
&=\Hom_{D_{w^C}}(\mc{O}_{C(w^C), \eta}, i_{w^C}^\bullet i_{w^C+} \mc{O}_{C(w^C), \eta})\\
&= \Hom_{D_{w^C}}(\mc{O}_{C(w^C), \eta}, \mc{O}_{C(w^C), \eta)}) \\
&= \C. 
\end{align*}
(Here we are writing $\mc{D}$-modules as objects in the derived category by considering them as complexes concentrated in degree $0$.) 
This guarantees that there is a canonical morphism 
\begin{equation}
    \label{canonical morphism}
    \mc{M}(w^C, \lambda, \eta) \rightarrow \mc{I}(w^C, \lambda, \eta)
\end{equation}
in $\mc{M}_{\text{coh}}(\mc{D}_\lambda, N, \eta)$. The image of this morphism is the irreducible module $\mc{L}(w^C, \lambda, \eta) \subset \mc{I}(w^C, \lambda, \eta)$. 

\subsection{Beilinson--Bernstein localization} 
The category of $\eta$-twisted Harish-Chandra sheaves is related to the category $\mc{N}$ via Beilinson--Bernstein localization. More precisely,  for $\lambda \in \h^*$ such that $\lambda + \rho$ is regular\footnote{We say $\mu \in \h^*$ is {\em regular} if $\alpha^\vee(\mu) \neq 0$ for all $\alpha \in \Sigma$.} and antidominant\footnote{We say $\mu \in \h^*$ is {\em antidominant} if $\alpha^\vee(\mu)\not\in \mathbb{Z}_{>0}$ for all $\alpha \in \Sigma$.} the Beilinson--Bernstein localization functor $\Delta_{\lambda+\rho}$ (defined by $\Delta_{\lambda + \rho}(V) = \mc{D}_{\lambda + \rho} \otimes_{U(\g)/\xi(\lambda)U(\g)} V$ for a $U(\g)/\xi(\lambda)U(\g)$-module $V$) provides an equivalence of categories:
\begin{equation}
\label{BB equivalence}
\Delta_{\lambda+\rho}:\mc{N}(\xi(\lambda), \eta) \xrightarrow{\sim} \mc{M}_{\text{coh}}(\mc{D}_{\lambda+\rho}, N, \eta)
\end{equation}
The inverse functor is given by global sections. The global sections of standard (resp. irreducible) $\eta$-twisted Harish-Chandra sheaves are standard (resp. irreducible) Whittaker modules.  

\begin{proposition}[{\cite[Thm. 9, Thm. 10]{Romanov}}]\label{prop:global-sections-of-standard} Recall that for a coset $C \in W_\eta \backslash W$, we denote by $w^C \in C$ the longest element. For $\lambda \in \h^*$ such that $\lambda + \rho$ is antidominant, 
\[
\Gamma(X,\mc{M}(w^C,\lambda +\rho,\eta))\cong M(w^C\cdot\lambda,\eta).
\]
If $\lambda+\rho \in \h^*$ is also regular, 
\[
\Gamma(X, \mc{L}(w^C, \lambda+\rho, \eta)) = L(w^C \cdot \lambda, \eta).
\]
\end{proposition}

Moreover, the localizations of standard Whittaker modules are standard $\eta$-twisted Harish-Chandra sheaves. We will need the following special case of this in the computations of Section \ref{lie algebra homology of standard whittaker modules}.

\begin{proposition}
\label{prop:localization}
Let $\lambda \in \h^*$, $\mu \in W \cdot \lambda$, $\eta \in \ch \mf{n}$ nondegenerate, and $w_0 \in W$ the longest element of the Weyl group. Then 
\[
\Delta_{\mu + \rho}(M(\lambda, \eta)) = \mc{M}(w_0, \mu + \rho, \eta).
\]
\end{proposition}

\begin{proof}
Using the adjunction $(\Delta_{\mu + \rho}, \Gamma)$,  \cite[Prop. 7]{Romanov}, and Proposition \ref{whittaker facts}(ii), we compute 
\begin{align*}
    \Hom_{\mc{M}_{\text{coh}}(\mc{D}_\lambda, N, \eta)}(\Delta_{\mu + \rho}&(M(\lambda, \eta)), \mc{M}(w_0, \mu + \rho, \eta)) \\
    &= \Hom_{\mc{N}(\xi(\lambda), \eta)}(M(\lambda, \eta), \Gamma(X, \mc{M}(w_0, \mu + \rho, \eta)))\\
    &= \Hom_{\mc{N}(\xi(\lambda), \eta)}(M(\lambda, \eta), M(w_0 \cdot \mu, \eta))\\
    &=\Hom_{\mc{N}(\xi(\lambda), \eta)}(M(\lambda, \eta), M(\lambda, \eta))\\
    &=\C.
\end{align*}
The category $\mc{M}_\text{coh}(\mc{D}_\lambda, N, \eta)$ is semisimple with one irreducible object, $\mc{M}(w_0, \mu + \rho, \eta)$ \cite[Theorem 5.5]{TwistedSheaves}, so we conclude that 
\[
\Delta_{\mu + \rho}(M(\lambda, \eta)) = \mc{M}(w_0, \mu + \rho, \eta).
\]
\end{proof}

\subsection{Geometric fibres and Lie algebra homology}

For an $\mc{O}_X$-module $\mc{F}$ on $X$, we denote by $T_x(\mc{F})$ its geometric fibre; i.e.
\[
T_x(\mc{F}) = \mc{F}_x/\mf{m}_x \mc{F}_x,
\]
where $\mf{m}_x$ is the maximal ideal corresponding to $x \in X$. This defines a right exact covariant functor 
\[
T_x: \mc{M}(\mc{O}_X) \rightarrow \Vect_\C.
\]
We can use the geometric fibre functor to compute Lie algebra in the following way. 
\begin{proposition}\label{prop:geometric-lie-homology}
Let $\lambda+\rho \in \mf{h}^*$ be regular, $w_0$ be the longest element of $W$, and $V$ be a $\mf{g}$-module with infinitesimal character $\chi^\lambda$. Then
\[
T_x ( \Delta_{\lambda+\rho} (V)) = H_0(\bar{\n}, V)_{w_0\cdot\lambda}, 
\]
for each $x \in C(w_0)\subset X$.
\end{proposition}
\begin{proof}
This follows from \cite[Ch. 3 \S2 Cor. 2.6]{Localization}. 
\end{proof}

%% file: sec-lie-algebra-homology-of-standard-whittaker-modules.tex
\section{Lie algebra homology of standard Whittaker modules}
\label{lie algebra homology of standard whittaker modules}

In this section we give an algebraic and geometric calculation of the $\bar{\n}$-coinvariants (Definition \ref{Lie algebra homology}) of standard Whittaker modules (Theorem \ref{thm:homology}). We begin with several preliminary algebraic results. 

Let $R, S$ be rings. We say that a $(R,S)$-bimodule $M$ is a free $(R,S)$-bimodule of rank $n$ if $M$ is a free $R\otimes S^{\text{op}}$-module of rank $n$. The following lemma is well-known; we include a standard proof which is adapted from \cite[Lem. 5.7]{TwistedSheaves}.

\begin{lemma}\label{U-free}
$U(\g)$ is a free $(U(\bar{\n}),Z(\g)\otimes U(\n))$-bimodule of rank $|W|$. 
\end{lemma}
\begin{proof}
Let $\{U_p(\g);p\in\mathbb{Z}_{\ge 0}\}$ denote the filtration of $U(\g)$ induced from the degree filtration of Poincare-Birkhoff-Witt basis elements. Let $F_pU(\g)$ be the linear filtration of $U(\g)$ defined by 
\[F_pU(\g)=U(\bar{\n})\otimes_\C U_p(\h)\otimes_\C U(\n).\]
The $F_p$ filtration of $U(\g)$ induces a filtration (which we again denote by $F_p$) on $Z(\g)\otimes U(\n)$: 
\[
F_p(Z(\g)\otimes U(\n)):=(U_p(\g)\cap Z(\g))\otimes U(\n). 
\]
The $F_p$ filtration of $Z(\g)\otimes U(\n)$ preserves the ring structure of $Z(\g)\otimes U(\n)$, i.e. if $u\in F_p(Z(\g)\otimes U(\n))$ and $u'\in F_q(Z(\g)\otimes U(\n))$, then $uu'\in F_{p+q}(Z(\g)\otimes U(\n))$. Therefore, the corresponding graded object $\text{gr}(Z(\g)\otimes U(\n))$ is a ring. Moreover, as rings, 
\[
\text{gr}(Z(\g)\otimes U(\n))\cong \text{gr}Z(\g)\otimes U(\n),
\]
where $\text{gr}Z(\g)$ denotes the graded ring associated to the filtration $U_p(\g)\cap Z(\g)$ of $Z(\g)$. It is well-known that the Harish-Chandra homomorphism $\gamma$ preserves the filtrations of $Z(\g)$ and $S(\h)$, and induces an isomorphism $\text{gr}\gamma$ from $\text{gr}Z(\g)$ to $S(\h)^W$. 

Suppose $z\in F_p(Z(\g))$. Then the Harish-Chandra homomorphism implies $z-\gamma(z)\in U_{p-1}(\g)\n$ and $\gamma(z)\in U_p(\h)$. Viewing $U(\g)$ as a right $Z(\g)\otimes U(\n)$-module, we have
\begin{align*}
U_q(\h)\cdot z &\subset  U_q(\h)\gamma(z)+U_q(\h)U_{p-1}(\g)\n \\
& \subset U_{q+p}(\h)+U_{q+p-1}(\g)\n\\
&\subset U_{q+p}(\h)+F_{q+p-1}U(\g) \subset F_{q+p}U(\g),
\end{align*}
because $U_{q+p-1}(\g)\subset F_{q+p-1}U(\g)$ and $\n\in F_0U(\g)$. This implies that $(U(\g),F_\bullet)$ is a filtered right $(Z(\g)\otimes U(\n),F_\bullet)$-module. Suppose $u\in U(\bar{\n})$. Then
\begin{align*}
    u\cdot F_pU(\g)&\subset u\cdot (U(\bar{\n})\otimes_\C U_p(\h)\otimes_\C U(\n))\\
    &\subset  F_pU(\g).
\end{align*}
Therefore, the left action of $U(\bar{\n})$ on $U(\g)$ preserves the $F_pU(\g)$ filtration of $U(\g)$. Altogether, we have shown that $U(\g)$ is a filtered $(U(\bar{\n}),Z(\g)\otimes U(\n))$-bimodule with the $F_\bullet$-filtration on $U(\g)$ and $Z(\g)\otimes U(\n)$, and the trivial filtration on $U(\bar\n)$.

Let $\text{gr}U(\g)$ denote the graded $(U(\bar{\n}),\text{gr}Z(\g)\otimes U(\n))$-bimodule corresponding to the $F_pU(\g)$ filtration of $U(\g)$. We have that 
\[\text{gr}U(\g)\cong U(\bar{\n})\otimes_\C U(\h)\otimes_\C U(\n),\] 
where $U(\bar{\n})$ acts by left multiplication, $\text{gr}Z(\g)=S(\h)^W$ acts on the $U(\h)$ factor, and $U(\n)$ acts by right multiplication. It is well-known that $U(\h)$ is a free $S(\h)^W$-module of rank $|W|$. Moreover, as $\C[W]$-modules, 
\[
U(\h)\cong \C[W]\otimes_\C S(\h)^W. 
\]
Therefore, $\text{gr}U(\g)$ is a free $(U(\bar{\n}),\text{gr}Z(\g)\otimes U(\n))$-bimodule of rank $|W|$. Let $\{\delta_w:w\in W\}$ be the canonical vector space basis for $\C[W]$, and identify each $\delta_w$ with an element $b_w$ of $U(\h)$ using the above isomorphism. Then \[\{1\otimes b_w\otimes 1 \in U(\bar{\n})\otimes_\C U(\h)\otimes_\C U(\n):w\in W\}\]
is a basis of $\text{gr}U(\g)$ as a free $(U(\bar{\n}),\text{gr}Z(\g)\otimes U(\n))$-bimodule. The multiplication map from the free $(U(\bar{\n}),Z(\g)\otimes U(\n))$-bimodule generated by $\{b_w\}_{w\in W}$ to $U(\g)$ is bijective by ~\cite[Ch.\ 3, \S 2, No.\ 8, Cor.\ 3]{BourbakiCommutativeAlgebra}. In other words, $U(\g)$ is a free $(U(\bar{\n}),Z(\g)\otimes U(\n))$-bimodule of rank $|W|$, with basis given by $\{b_w\}$. 
\end{proof}
\begin{lemma}\label{lemma:weight-space}
Let $\lambda+\rho \in \h^*$ be regular\footnote{Here, and in what follows, we assume regularity only to simplify statements and proofs of the results; we expect analogous statements to hold more generally, without much additional difficulty.} and $\xi(\lambda)=\ker\chi^\lambda \in \Max Z(\g)$ the maximal ideal corresponding to $\lambda$. Define $V^\lambda := U(\h)/p(\xi(\lambda))U(\h)$, where $p:Z(\g)\rightarrow U(\h)$ is the Harish-Chandra homomorphism. Then, as $U(\h)$-modules, 
\[
V^\lambda = \bigoplus_{w \in W} \C_{w\cdot\lambda}.
\]
\end{lemma}
\begin{proof}
We adapt the proof of \cite[Ch.3 \S2 Lem. 2.2]{Localization} to fit our setting. It follows from standard properties of the Harish-Chandra isomorphism that $\dim V^\lambda=|W|$ and
\[
p(\xi(\lambda))U(\h)\subset \ker\mu \Longleftrightarrow \mu\in W\cdot\lambda,
\]
where $\ker\mu$ is the kernel of $\mu:U(\h)\rightarrow \C$. Therefore, the quotient map
\[
q:V^\lambda\rightarrow \C_\mu = U(\h)/\ker\mu
\]
is well defined and nonzero for each $\mu\in W\cdot\lambda$. For $h\in\h$, multiplication by $h-\mu(h)$ defines a map $m:V^\lambda\rightarrow V^\lambda$. Clearly, $q\circ m =0$. Because $q$ is nonzero, $m$ must not be surjective. Because $V^\lambda$ is finite-dimensional and $m$ is not surjective, we conclude that $\ker m \neq 0$. Therefore, $\mu(h)$ is an eigenvalue of the operator $h\in\h$ on $V^\lambda$ if $\mu\in W \cdot \lambda$. By symmetry, each eigenvalue has the same multiplicity. Because $\lambda+\rho$ is regular, each eigenvalue is distinct, and the result follows. 
\end{proof}
\begin{theorem}\label{thm:homology}
Let $\lambda+\rho$ be regular. As $U(\mathfrak{h})$-modules,
\[
H_0(\bar{\n},M(\lambda,\eta)) = \bigoplus_{w\in W_\eta}\mathbb{C}_{w\cdot\lambda}.
\]
\end{theorem}

We include both an algebraic and geometric proof of this theorem. 

% Proof idea: Let's think about $Y(\lambda,\eta)$ (for $\mathfrak{sl}_2$) first. The $\bar{\n}$-homology space is 
% \[
% \mathbb{C}\otimes_{\bar{\n}}U(\g)\otimes_{Z(\g)\otimes U(\n)}\mathbb{C}_{\chi^\lambda,\eta}.
% \]
% The idea is to write something in $U(\g)$ using the PBW basis $f^i h^j e^k$. Any basis element with an $f$ on the left will get killed in the tensor product. Similarly, any basis with an $e$ on the right can be turned into a basis vector without an $e$ in the tensor product. So the only basis elements left are the $h$'s. Now if I have an element $H$ in $U(\h)$ which is in the image of the HC homomorphism $p$, then I can write $Z = H + FE$, where $Z\in Z(\g)$, and $FE$ is a PBW basis element with $f$'s on the left. In the tensor product, $1\otimes H\otimes 1 = 1\otimes Z\otimes 1$, because the $FE$ term gets killed in the tensor. From there, I can pull the $Z$ through on the right: $1\otimes Z\otimes 1 = 1\otimes 1 \otimes \chi^\lambda(Z)$. So the only elements of $U(\g)$ which can't be simplified in the tensor product are those in $U(\h)$ which are not in $p(\chi^\lambda)$. In other words, $H_0(\bar{\n},Y(\lambda,\eta)) = U(\h)/p(\chi_\lambda)$. From here, there should be some standard Bourbaki fact that this space decomposes into a $W$-orbit of weight spaces. 

\vspace{2mm}
\noindent
{\em Algebraic proof.}
First, using the Poincare-Birkhoff-Witt basis of $U(\g)$, we have:
\[U(\g) = \bar{\n}^\eta U(\g)\oplus U(\p_\eta).\]
Therefore, as $U(\h)$-modules, we have
\begin{align*}
       H_0(\bar{\n},M(\lambda,\eta)) &= \C\otimes_{\bar{\n}} U(\g)\otimes_{\p_\eta} Y_{\lev_\eta}(\lambda,\eta) \\
       &= \C\otimes_{\bar{\n}} (\bar{\n}^\eta U(\g)\oplus U(\p_\eta))\otimes_{\p_\eta} Y_{\lev_\eta}(\lambda,\eta)\\
    &=\C\otimes_{\bar{\n}_\eta} Y_{\lev_\eta}(\lambda,\eta)\\
&= H_0(\bar{\n}_\eta, Y_{\lev_\eta}(\lambda,\eta)).
\end{align*}
The character $\eta\vert_{\n_\eta}\in \ch\n_\eta$ is nondegenerate, and $Y_{\lev_\eta}(\lambda,\eta)$ is a nondegenerate Whittaker module. Hence it suffices to prove the result for $\eta$ nondegenerate. In this setting, $M(\lambda,\eta)=Y(\lambda,\eta)$, and we have a surjective $U(\h)$-module morphism
\begin{align*}
    U(\h)&\rightarrow H_0(\bar{\n},Y(\lambda,\eta))= \C\otimes_{\bar{\n}} U(\g)\otimes_{Z(\g)\otimes U(\n)}\C_{\chi^\lambda,\eta}\\
    H&\mapsto 1\otimes H\otimes 1.
\end{align*}
By the Casselman--Osborne Lemma~\cite[Lem. 2.5]{CasselmanOsborne}, if $z\in \xi(\lambda)$, then $p(z)$ annihilates $H_0(\bar{\n},Y(\lambda,\eta))$. Therefore, the above surjective homomorphism of $U(\h)$-modules factors through $V^\lambda=U(\h)/p(\xi(\lambda))U(\h)$:
\begin{center}
\begin{tikzpicture}
\node at (0,0) {$V^\lambda$};
\node at (-2,1) {$U(\h)$};
\node at (2,1) {$H_0(\bar{\n},Y(\lambda,\eta))$};
\draw[->] (-1.5,1)--(.8,1);
\draw[->] (-1.5,.8)--(-.3,.15);
\draw[->] (.25,.15)--(1.3,.7);
\end{tikzpicture}
\end{center}
Moreover, by Lemma \ref{lemma:weight-space}
\[
V^\lambda \cong \bigoplus_W\C_{w\cdot \lambda}. 
\]
Because $U(\g)$ is a free $(U(\bar{\n}),Z(\g)\otimes U(\n))$-bimodule of rank $|W|$ (Lemma \ref{U-free}), we have that
\[
\dim H_0(\bar{\n},Y(\lambda,\eta)) = |W|. 
\]
Therefore, the surjective map $V_\lambda \rightarrow H_0(\bar{\n},Y(\lambda,\eta))$ is an isomorphism of $U(\h)$-modules, and the theorem follows. \qed

\vspace{2mm}
\noindent
{\em Geometric proof.}
 We include a geometric proof for the case when $\eta$ is nondegenerate, which, by the Poincare-Birkhoff-Witt theorem and the initial argument of the algebraic proof, implies the general result. 
 
Assume $\eta \in \ch{\mf{n}}$ is nondegenerate. By \cite[Cor. 2.4]{Localization}, the only possible $\h$-weights of $H_0(\overline{\mf{n}}, M(\lambda, \eta))$ are $\mu \in W \cdot \lambda$. For $\mu \in W\cdot \lambda$, we will can compute $H_0(\overline{\mf{n}}, M(\lambda, \eta))$ using Proposition \ref{prop:geometric-lie-homology} and Proposition \ref{prop:localization}. Indeed, for $x \in C(w_0)$, we have 
 \begin{align*}
     H_0(\overline{\mf{n}}, M(\lambda, \eta))_\mu &= T_x(\Delta_{w_0 \cdot \mu + \rho}(M(\lambda, \eta))) \\
     &= T_x(\mc{M}(w_0, w_0 \cdot \mu + \rho, \eta)).
 \end{align*}
 Because $\mc{M}(w_0, w_0 \cdot \mu + \rho, \eta) := i_{w_0!}(\mc{O}_{C(w_0), \eta})$, \[
\dim T_x(\mc{M}(w_0, w_0 \cdot \mu + \rho, \eta)) = 1.
 \]
 This implies the result. \qed

 %Because we only require computing the degree zero $\bar{\n}$-homology of the Whittaker module $M(\lambda,\eta)$, by Proposition \ref{prop:geometric-lie-homology}, it suffices to compute $T_x\circ\Delta_\mu(M(\lambda,\eta))$ for $x\in C(w_0)\subset X$ and each $\mu\in\h$. By Corollary \ref{cor:localization}, we have 
%\[
%\Delta_{\mu}(M(\lambda,\eta))=\begin{cases}
%\mc{M}(w_0,\mu,\eta)&\text{ for }\mu\in W\cdot\lambda,\\
%0&\text{ else.}
%\end{cases}
%\]
%Recall that $\mathcal{M}(w_0,\mu,\eta):= i_{w_0!}(\cO_{C(w_0),\eta})$. Therefore, 
%\[
%\dim T_x ( \mathcal{M}(w_0,\mu,\eta)) =1,
%\]
%for each $x\in C(w_0)\subset X$. Combining these results with Proposition \ref{prop:geometric-lie-homology}, we see that
%\[
%\dim H_0(\bar{\n},M(\lambda,\eta))_\mu = \begin{cases}
%1\quad &\text{if }\mu\in W\cdot\lambda,\\
%0&\text{else.}
%\end{cases}
%\]
%The theorem then follows from the fact that $H_0(\bar{\n},M(\lambda,\eta))$ is semisimple as an $\h$-module for $\lambda$ regular (cf. \cite[Corollary 2.4]{Localization}). 

%% file: sec-contravariant-pairings.tex
\section{Contravariant pairings of standard Whittaker modules}
\label{contravariant pairings of standard whittaker modules}

In this section, we define and classify contravariant pairings between standard Whittaker modules and Verma modules. Contravariant pairings play an analogous role for $\N$ as contravariant forms for category $\cO$. We prove that contravariant pairings are unique up to scalar multiple, and that $M(\lambda, \eta)$ admits a nonzero contravariant pairing with a Verma module of highest weight $\mu$ if and only if $\mu\in W_\eta \cdot \lambda$. We give an explicit construction of a contravariant pairing between $M(\lambda, \eta)$ and $M(w \cdot \lambda)$ for each $w \in W_\eta$. This construction degenerates to the Shapovalov form on a Verma module when $\eta=0$. We finish by describing properties of the (left) radical of a contravariant pairing. 

Let $\{y_\alpha, x_\alpha\}_{\alpha \in \Sigma^+} \cup \{h_\alpha\}_{\alpha \in \Pi}$ be a Chevalley basis of $\g$ with $x_\alpha \in \g_\alpha$, $y_\alpha \in \g_{-\alpha}$ and $h_\alpha \in \h$ such that $[x_\alpha, y_\alpha]=h_\alpha$. Let $\tau:U(\g)\rightarrow U(\g)$ be the transpose antiautomorphism defined by $\tau(x_\alpha)=y_{\alpha}$ and $\tau(h_\alpha)=h_\alpha$. Let $M(\mu)$ be the Verma module of highest weight $\mu \in \h^*$ and $M(\lambda, \eta)$ the standard Whittaker module (Definition \ref{standard Whittaker module}) corresponding to $(\lambda, \eta) \in \h^* \times \ch\n$.  \begin{definition}
\label{contravariant pairing def}
A \emph{contravariant pairing} between $\mf{g}$-modules $V$ and $W$ is a bilinear pairing
\[
\langle \cdot , \cdot \rangle:V \times W \rightarrow\mathbb{C}
\]
such that
\[
\langle u v,w\rangle = \langle v,\tau(u)w\rangle
\]
for all $u\in U(\g)$, $v\in V$, and $w \in W$.
\end{definition}
\begin{theorem}\label{thm:unique-forms}
Assume $\lambda+\rho\in \h^*$ is regular. There exists a nonzero contravariant pairing between $M(\lambda,\eta)$ and $M(\mu)$ if and only if $\mu \in W_\eta\cdot\lambda$. Moreover, a contravariant pairing between $M(\lambda,\eta)$ and $M(\mu)$ is unique up to scalar multiple. 
\end{theorem}
\begin{proof}
Let $\Psi$ be the space of contravariant pairings between $M(\lambda,\eta)$ and $M(\mu)$. We have the following canonical isomorphism
\begin{align*}
   \Psi&\rightarrow \Hom_\g (M(\mu),M(\lambda,\eta)^\ast)\\
   \langle \cdot, \cdot\rangle&\mapsto \phi(y)(\cdot):= \langle  \cdot,y \rangle.
\end{align*}
Moreover,
\begin{align*}
\Hom_\g (M(\mu),M(\lambda,\eta)^\ast) & = \Hom_\mathfrak{b}(\mathbb{C}_{\mu}, M(\lambda,\eta)^\ast)\\
&= H^0(\n, M(\lambda,\eta)^\ast)_{\mu}\\
& = H_0(\bar{\n},M(\lambda,\eta))^\ast_{\mu}.
\end{align*}
The last equality above follows from tensor-hom adjunction and the $\g$-module structure on $X^\ast$ (which accounts for the duality between $\n$ and $\bar{\n}$), where we identify $H^0(\n,X^\ast)$ as $\Hom_\n(\C,\Hom_\C(X,\C))$ and $H_0(\bar{\n},X)$ as $X\otimes_{\bar{\n}}\C$, see Definition \ref{Lie algebra homology} and \ref{twisted Lie algebra cohomology}. The result then follows from Theorem~\ref{thm:homology}. 
\end{proof}

We will now give an explicit construction of a nonzero contravariant pairing between $M(\lambda,\eta)$ and $M(w\cdot\lambda)$ for $w\in W_\eta$ which generalizes the Shapovalov form on Verma modules. We start with some preparatory results. 

\begin{lemma}\label{lemma:twisted-HCproj}
Let $\ker\eta$ be the kernel of $\eta:U(\n)\rightarrow\C$. There is a direct sum decomposition
\[
U(\g)=U(\h)\oplus (\bar{\n}U(\g)+U(\g)\ker\eta).
\]
Denote by $\pi_\eta:U(\g)\rightarrow U(\h)$ projection onto the first coordinate in this decomposition. 
\end{lemma}
\begin{proof} Choose an order on the set of roots so that 
\[
\{ y^{\bar{I}} h^J x^K:= y_{\alpha_n}^{i_n} \cdots y_{\alpha_1}^{i_1} h_{\alpha_1}^{j_1} \cdots h_{\alpha_r}^{j_r} x_{\alpha_1}^{k_1} \cdots x_{\alpha_n}^{k_n} \mid i_\ell, j_\ell, k_\ell \in \Z_{\geq 0} \}
\]
is a Poincar\'{e}--Birkhoff--Witt basis of $U(\g)$. Here $I=(i_1, \ldots, i_n)$, $J=(j_1, \ldots, j_r)$, and $K=(k_1, \ldots, k_n)$ are multi-indices, $\bar{I}=(i_n, \ldots, i_1)$, and $y=(y_{\alpha_n}, \ldots , y_{\alpha_1})$, $h=(h_{\alpha_1}, \ldots, h_{\alpha_r})$, $x=(x_{\alpha_1}, \ldots, x_{\alpha_n})$. We can write $x^{\bar{I}}h^Jx^K$ as
\[
y^{\bar{I}}h^Jx^K = y^{\bar{I}}h^J (x^K-\eta(x^K)) + \eta(x^K)y^{\bar{I}}h^J. 
\]
The terms $y^{\bar{I}}h^J (x^K-\eta(x^K))$ and $ \eta(x^K)y^{\bar{I}}h^J$ are elements in $U(\h)  +\bar{\n}U(\g)+ U(\g) \ker \eta$. By extending linearly, we can write any element of $U(\g)$ as a sum of a vector in $U(\h)$, a vector in $\bar{\n} U(\g)$ and a vector in $U(\g) \ker \eta$. The intersection 
\[
U(\h) \cap (\bar{\n} U(\g) + U(\g) \ker \eta) = 0,
\]
so the sum is direct.
\end{proof}

\begin{lemma}\label{lemma:induction-annihilator}
Let $\mf{p}$ be a parabolic subalgebra of $\mf{g}$ and $\mf{l}$ the corresponding reductive Levi subalgebra. The Lie algebra $\mf{p}$ decomposes as $\mf{p} = \mf{l} \oplus \mf{n}$ for a nilpotent subalgebra $\mf{n}$ of $\mf{p}$. Let $N$ be a $U(\mf{l})$-module and $M$ the $U(\mf{g})$-module parabolically induced from $N$,
\[
M=U(\mf{g}) \otimes_{U(\mf{p})} N
\]
where $N$ is considered as a $U(\mf{p})$-module via the natural projection map $\mf{p} \rightarrow \mf{l}$. For $x\in N$ and $v=1\otimes x \in M$, 
\[
\Ann_{U(\mf{g})} v = U(\g)\n+ U(\g) \Ann_{U(\mf{l})} x.
\]

\end{lemma}
\begin{proof}
By definition of $M$, we have
\[
U(\g)\n+U(\g)\Ann_{U(\mf{l})}x \subseteq \Ann_{U(\g)} v.
\]
We will prove the reverse inclusion. Let $u \in \Ann_{U(\g)}v$. Then $u \otimes x = 0$, so $u=u'a$ for $u' \in U(\g)$ and $a \in \Ann_{U(\mf{p})}x$; i.e. 
\[
\Ann_{U(\g)}v \subseteq U(\g) \Ann_{U(\p)}x.
\]

Since $\n$ acts trivially on $N$ and we have a PBW decomposition $U(\p) = U(\mf{l}) \otimes U(\n)$, 
\[
\Ann_{U(\p)}x = U(\p) \n + U(\p) \Ann_{U(\mf{l})}x.
\]
Hence the reverse inclusion holds. 
%By the Poincare-Birkhoff-Witt theorem
%\begin{equation}
%    \label{parabolic decomp}
%    U(\g) = U(\bar{\n})\otimes U(\p_\eta).
%\end{equation}

%Because $M$ is a free $U(\bar{\n})$-module, any element $u\in U(\g)$ which annihilates $v$ must be contained in $U(\bar{\n}) \Ann_{U(\p_\eta)}v$. (If not, then $uv$ must be contained in the $U(\bar{\n})$-span of a nonzero vector $1\otimes t \in M$, and therefore can not be zero). By definition of the action of $\p_\eta$ on $M$, we have $\Ann_{U(\p_\eta)} v = U(\p_\eta)\Ann_{U(\mf{l})}v + U(\p_\eta)\n^\eta$. Therefore, $\Ann_{U(\g)}v \subset U(\g)\n^\eta+U(\g)\Ann_{U(\mf{l}_\eta)}v$.
\end{proof}
\begin{proposition}[{\cite[Thm. 3.1]{Kostant}}]\label{prop:kostant-annihilator}
Let $\xi(\lambda)=\ker\chi^\lambda\in\Max Z(\g)$, $\eta\in\ch\n$ be nondegenerate, and $\omega\in Y(\lambda,\eta)$ be a nonzero Whittaker vector. Then 
\[
\Ann_{U(\g)}\omega = U(\g)\xi(\lambda)+U(\g)\ker\eta,
\]
where $\ker\eta$ is the kernel of the map $\eta:U(\n)\rightarrow \C$. 
\end{proposition}
\begin{corollary}\label{corollary:annihilators}
Let $\eta \in \ch \n$ be arbitrary and recall the map $\xi_\eta:\h^\ast\rightarrow\Max Z(\lev_\eta)$. Let $v=1\otimes 1\in M(w\cdot\lambda)$ and $\omega = 1\otimes 1\otimes 1 \in M(\lambda,\eta)$. We have
\begin{align*}
 \Ann_{U(\g)}v & = U(\g)\n+U(\g)\ker(w\cdot\lambda),\text{ and }\\
 \Ann_{U(\g)}\omega & = U(\g)\ker\eta+U(\g)\xi_\eta(\lambda).
\end{align*}
\end{corollary}
\begin{proof}
For the first equality, we apply Lemma \ref{lemma:induction-annihilator} to the vector $v = 1\otimes 1\in M(w\cdot\lambda) = U(\g)\otimes_\bb\C_{w\cdot\lambda}$. By definition of $\C_{w\cdot\lambda}$, we have $\Ann_{U(\h)} x = \ker w\cdot\lambda$. Therefore, $\Ann_{U(\g)} v = U(\g)\n+U(\g)\ker (w\cdot\lambda)$. 

For the second equality, we apply Lemma \ref{lemma:induction-annihilator} to $\omega = 1\otimes 1\otimes 1 \in M(\lambda,\eta)$. By Lemma \ref{lemma:induction-annihilator} and Proposition \ref{prop:kostant-annihilator}, we have
\[
\Ann_{U(\g)} \omega = U(\g)\n^\eta + U(\g)(U(\lev_\eta)\ker\eta\vert_{\n_\eta}+ U(\lev_\eta)\xi_\eta(\lambda)). 
\]
Clearly, $U(\g)U(\lev_\eta)\ker \eta\vert_{\n_\eta} = U(\g)\ker \eta\vert_{\n_\eta} $ and $U(\g)U(\lev_\eta)\xi_\eta(\lambda) = U(\g)\xi_\eta(\lambda) $. Moreover, $U(\g)\n^\eta+ U(\g)\ker\eta\vert_{\n_\eta}= U(\g)\ker\eta$. Therefore, 
\[
\Ann_{U(\g)}\omega = U(\g)\ker\eta+U(\g)\xi_\eta(\lambda).
\]
\end{proof}

Let $\omega = 1\otimes 1 \otimes 1 \in M(\lambda,\eta)$ and $v = 1\otimes 1 \in M(w\cdot\lambda)$. For $w \in W_\eta$, define a bilinear pairing between $M(\lambda,\eta)$ and $M(w\cdot\lambda)$ by
\begin{equation}
    \label{contravariant pairing}
\langle x,y\rangle_w:= ((w\cdot\lambda)\circ \pi_\eta)(\tau(u')u),
\end{equation}
where $u, u' \in U(\g)$ are such that $x=u \omega$ and $y = u'v$, and $\pi_\eta$ is the projection map defined in Lemma \ref{lemma:twisted-HCproj}. (Note that the choices of $u, u'$ are not necessarily unique.) 

\begin{theorem}\label{thm:shapovalov}
The bilinear pairing $\langle\cdot,\cdot\rangle_w:M(\lambda,\eta)\times M(w\cdot\lambda)\rightarrow\C$ of equation \ref{contravariant pairing} is well-defined, nonzero, and contravariant. When $\eta=0$, $\langle\cdot,\cdot\rangle_w$ is the Shapovalov form. 
\end{theorem}
\begin{proof}
The bilinear pairing $\langle\cdot,\cdot\rangle_w$ is contravariant by construction. Moreover, $\langle \omega, v\rangle_w = 1$ for $\omega=1\otimes 1 \in M(\lambda,\eta)$ and $v=1\otimes1 \in M(w\cdot\lambda)$. Therefore, the pairings are nonzero. When $\eta = 0$, $W_\eta = 1$, and $\langle \cdot, \cdot \rangle_1: M(\lambda) \times M(\lambda) \rightarrow \C$ is the contravariant form defined by $\langle uv, u'v \rangle_1= \lambda \circ p(\tau(u')u)$, where $p$ is the Harish-Chandra homomorphism. This is exactly the Shapovalov form \cite[Eq. (5)]{Shapovalov}.

We now prove the result for $\eta\neq 0$. Recall the map $\pi_\eta:U(\g)\rightarrow U(\h)$ is given by projection onto the first coordinate in the decomposition $U(\g)=U(\h)\oplus (\bar{\n}U(\g)+U(\g)\ker\eta)$. To show that the pairing is well-defined, we must prove that $\langle x,y\rangle_w$ is independent of the choice of $u,u'\in U(\g)$ such that $x=u\omega$ and $y=u'v$. Choose  $\tilde{u},\tilde{u}'\in U(\g)$ so that $y=u'v = \tilde{u}'v$ and $x=u\omega= \tilde{u}\omega$. To establish that $\langle \cdot, \cdot \rangle_w$ is well-defined, we need to check that 
\[
(w \cdot \lambda) \circ  \pi_\eta (\tau (u' - \tilde{u}')u)=0 \text{ and } (w \cdot \lambda) \circ \pi_\eta (\tau(u')(u-\tilde{u})) = 0.
\]
As $\tau(u' - \tilde{u}')u = \tau( \tau(u)(u' - \tilde{u}')) \in \tau(\Ann_{U(\g)} v)$ and $\tau(u')(u - \tilde{u}) \in \Ann_{U(\g)}\omega$, it suffices to show that
\[
    ((w\cdot\lambda)\circ \pi_\eta)(\tau(\Ann_{U(\g)}v))=0,\text{ and }
    ((w\cdot\lambda)\circ \pi_\eta)(\Ann_{U(\g)}\omega)=0.
\]
Using Corollary \ref{corollary:annihilators}, this reduces to showing 
\begin{align}
\label{first}
     ((w\cdot\lambda)\circ \pi_\eta)(\bar{\n}U(\g))&=0;\\
\label{second}
     ((w\cdot\lambda)\circ \pi_\eta)(\ker(w\cdot\lambda)U(\g))&=0;\\
\label{third}
    ((w\cdot\lambda)\circ \pi_\eta)(U(\g)\ker\eta)&=0;\\
\label{fourth}
    ((w\cdot\lambda)\circ \pi_\eta)(U(\g)\xi_\eta(\lambda))&=0.
\end{align}
Equalities (\ref{first}) and (\ref{third}) are obvious. Recall that $p_\eta:U(\mf{l}_\eta)\rightarrow U(\h)$ is the Harish-Chandra homomorphism of $U(\mf{l}_\eta)$; i.e. projection onto the first coordinate in the decomposition $U(\mf{l}_\eta) = U(\h) \oplus (\bar{\mf{n}}_\eta U(\mf{l}_\eta) + U(\mf{l}_\eta) \n_\eta)$. To establish  (\ref{fourth}), we first note that
\[
\pi_\eta(\xi_\eta(\lambda)) = p_\eta(\xi_\eta(\lambda)) \subset \ker w \cdot \lambda
\]
for $w\in W_\eta$.
%This follows from the fact that $\xi_\eta(\lambda) \subset Z(\mf{l}_\eta) \subset U(\h) \oplus (\bar{\n} U(\g))$ (so $p_\eta$ and $p_0^\eta$ agree on $\xi_\eta(\lambda)$), the definition of $\xi_\eta(\lambda):= \ker(\lambda \circ p_0^\eta)|_{Z(\mf{l}_\eta)}$, and the fact that $\xi_\eta(w \cdot \lambda) = \xi_\eta(\lambda)$ for all $w \in W_\eta$. \textcolor{red}{This is nitpicking, but I think that the fact $\xi_\eta(w \cdot \lambda) = \xi_\eta(\lambda)$ for all $w \in W_\eta$ is proved by showing that $p_0^\eta(\xi_\eta(\lambda)) \subset \ker w \cdot \lambda$. So maybe this reasoning is a bit circular. How about this: } 
The first equality follows from the fact that $\xi_\eta(\lambda) \subset Z(\mf{l}_\eta)$ and $\pi_\eta$ and $p_\eta$ agree on $Z(\mf{l}_\eta)$. (Indeed, let $z \in Z(\mf{l}_\eta)$. By the Poincar\'{e}--Birkhoff--Witt theorem, $z \in U(\mf{h}) \oplus \overline{\mf{n}}_\eta U(\mf{l}_\eta)$, so we can express $z$ as $z=h+xu$ for $h \in U(\mf{h}), x \in \bar{\mf{n}}_\eta$, and $u \in U(\mf{l}_\eta)$, and the element $h \in U(\mf{h})$ in this decomposition is uniquely determined by $z$. As $xu \in \bar{\mf{n}} U(\mf{g}) + U(\mf{g}) \ker{\eta}$, we have $\pi_\eta(z) = h = p_\eta(z)$.) The inclusion follows from well-known properties of the Harish-Chandra homomorphism for $Z(\lev_\eta)$. 

Using the decomposition $U(\g) = U(\mf{l}_\eta) \oplus (\bar{\n}^\eta U(\g) + U(\g) \n^\eta)$, 
we have that 
\[
\pi_\eta(U(\g) \xi_\eta(\lambda)) = \pi_\eta(U(\mf{l}_\eta) \xi_\eta(\lambda)) + \pi_\eta(U(\g) \n^\eta \xi_\eta(\lambda)). 
\]
By \cite[Lem. 1.7]{McDowell}, $[\mf{l}_\eta, \n^\eta] \subset \n^\eta$. This implies that $[U(\mf{l}_\eta), \n^\eta] \subset U(\mf{l}_\eta) \n^\eta$, hence $[\xi_\eta(\lambda), \n^\eta] \subset U(\g) \ker \eta$. From this we conclude that $\pi_\eta(U(\mf{g}) \n^\eta \xi_\eta(\lambda))=0$. 

An application of Lemma \ref{lemma:twisted-HCproj} to $\mf{l}_\eta$ yields the decomposition 
\[
U(\mf{l}_\eta) = U(\h) \oplus (\bar{\n}_\eta U(\mf{l}_\eta) + U(\mf{l}_\eta) \ker \eta|_{U(\n_\eta)}).
\]
Because $\xi_\eta(\lambda)$ commutes with $\ker \eta|_{U(\n_\eta)}$, we can use this decomposition to conclude that 
\[
\pi_\eta(U(\mf{l}_\eta) \xi_\eta(\lambda)) = \pi_\eta(U(\h) \xi_\eta(\lambda)). 
\]
Finally, because $Z(\lev_\eta)\subset U(\lev_\eta)_0 = \{u\in U(\lev_\eta):[h,u]=0\text{ for all } h\in \h\}$, we have
\begin{align*}
\pi_\eta(U(\h)\xi_\eta(\lambda))&= U(\h)\pi_\eta(\xi_\eta(\lambda))\subset \ker w\cdot\lambda. 
\end{align*}
This proves (\ref{fourth}).

It remains to show (\ref{second}). Because $\pi_\eta$ is the identity on $\ker w\cdot \lambda$, we have 
\[
\pi_\eta(\ker w\cdot\lambda)\subset\ker w \cdot \lambda. 
\]
Moreover, because $[\h,\bar{\n}]\subset \bar{\n}$, we have $\pi_\eta(\ker(w\cdot\lambda)\bar{\n}U(\g))=0$. Using the decomposition in Lemma \ref{lemma:twisted-HCproj}, we conclude that 
\[
\pi_\eta(\ker (w\cdot \lambda )U(\g)) = \pi_\eta(\ker (w\cdot \lambda )U(\h))=\pi_\eta(U(\h)\ker (w \cdot\lambda)). 
\]
Again, because $U(\h)\subset U(\g)_0$, we have
\begin{align*}
\pi_\eta(U(\h)\ker w\cdot\lambda) &= U(\h)\pi_\eta(\ker w\cdot \lambda)\subset \ker w \cdot \lambda. 
\end{align*}
This proves (\ref{second}), and the proposition.
\end{proof}
Combining Theorem \ref{thm:unique-forms} and Theorem \ref{thm:shapovalov}, we see that if $\langle\cdot,\cdot\rangle$ is a nonzero contravariant pairing between $M(\lambda,\eta)$ and $M(\mu)$, then $\langle\cdot,\cdot\rangle$ is a scalar multiple the pairing $\langle\cdot,\cdot\rangle_w$ for some $w\in W_\eta$.

\begin{corollary}\label{cor:normalized-w-form}
Assume $\lambda+\rho\in\h^\ast$ is regular. Any contravariant pairing $\langle \cdot, \cdot \rangle:M(\lambda, \eta) \times M(\mu) \rightarrow \C$ is uniquely determined by $\langle \omega , v\rangle$, where $\omega$ and $v$ are the generating Whittaker vectors in $M(\lambda,\eta)$ and $M(\mu)$, respectively. 
\end{corollary}

\begin{proposition}\label{prop:radicals}
Let $\lambda\in\h^\ast$ and $\langle \cdot, \cdot \rangle_w$ be the contravariant pairing between $M(\lambda,\eta)$ and $M(w\cdot\lambda)$ defined by equation \ref{contravariant pairing}, for $w\in W_\eta$. 
\begin{enumerate}
    \item If $\nu,\gamma\in\z^\ast$, $\nu\neq\gamma$, $x\in M(\lambda,\eta)_{\nu}$, and $y\in M(w\cdot \lambda)_{\gamma}$, then $\langle x,y\rangle_w=0$.
    \item The left radical $\text{Rad}^L\langle \cdot, \cdot \rangle_w:=\{v\in M(\lambda,\eta):\langle v, M(w\cdot\lambda)\rangle_w=0\}$ is the maximal proper submodule of $M(\lambda,\eta)$. 
    % \item the right radical $\text{Rad}^R\langle,\rangle_w:=\{v\in M(\mu): \langle M(\lambda,\eta),v\rangle_w=0\}$ is the submodule generated by highest weight vectors with weights not contained in $W_\eta\lambda$ 
    % \[
    % \text{Rad}^R\langle,\rangle = U(\g)\left(\bigoplus_{y\in W- W_\eta} H^0(\n,M(\mu))_{y\circ\lambda}\right),
    % \] and

\end{enumerate}
\end{proposition}
\begin{proof}
The proof of (1) is identical to the standard category $\cO$ proof, which is as follows. Suppose $x\in M(\lambda,\eta)_{\nu}$ and $y\in M(w\cdot\lambda)_{\gamma}$ with $\nu\neq \gamma$ as characters of $\z$. By the contravariance of the bilinear pairing, for each $z\in\z$ (recall that $\tau(z)=z$), we have 
\begin{align*}
    \nu(z)\langle x, y\rangle_w &= \langle z x, y\rangle_w = \langle  x, zy\rangle_w =\gamma(z)\langle  x, y\rangle_w.
\end{align*}
Therefore, $\langle  x, y\rangle_w =0$. 

We will now prove (2). If $\eta$ is nondegenerate, or, more generally, if $M(\lambda, \eta)$ is irreducible, then $\text{Rad}^L\langle \cdot, \cdot \rangle_w = 0$. Assume that $\eta$ is degenerate (i.e. $\eta$ vanishes on at least one simple root space) and $M(\lambda,\eta)$ is reducible. Let $N$ be the maximal proper submodule of $M(\lambda,\eta)$ and $x$ a nonzero $\z$-weight vector in $N$ with weight $\mu$. By \cite[Thm. 2.5]{McDowell}, $\mu \neq \lambda$ as characters of $\z$. The generating Whittaker vector (i.e. the highest weight vector) $v$ of $M(w\cdot\lambda)$ has $\z$-weight $w\cdot\lambda$. Moreover, because $w\in W_\eta$, we have that $w\cdot\lambda = \lambda$ as characters of $\z$ (recall that $M(w\cdot\lambda,\eta)\cong M(\lambda,\eta)$ for each $w\in W_\eta$). Therefore, (1) implies 
\[
\langle x, v\rangle_w = 0.
\]
Because $\mf{z}$ acts semisimply on $N$ and the pairing is bilinear, we have $\langle N, v\rangle_w=0$. 

If $y\in M(w\cdot\lambda)$, then there exists $u\in U(\g)$ such that $y = uv$. Therefore, for any $x\in N$,
\[
\langle x, y\rangle_w = \langle x, uv \rangle_w = \langle \tau(u)x, v\rangle_w = 0,
\]
because $\tau(u) x \in N$. Therefore, $N\subseteq \text{Rad}^L\langle \cdot, \cdot \rangle_w$. The result then follows from the fact that the pairing $\langle\cdot,\cdot\rangle_w$ is nonzero and the left radical of the pairing is a submodule of $M(\lambda,\eta)$.

\end{proof}

\begin{remark}
\label{rmk:contravariant-pairings-via-whittaker-functors}
In \cite[\S 3.2]{Matumoto}, Matumoto uses the Shapovalov form to define a contravariant pairing between an irreducible module in category $\cO$ and its completion. A similar construction applies to Verma modules: because weight spaces of $M(w \cdot \lambda)$ are orthogonal with respect to the Shapovalov form, the Shapovalov form extends to a contravariant pairing between the completed Verma module $\overline{M(w \cdot \lambda)}$ and $M(w \cdot \lambda)$. The Whittaker functors of Section \ref{sec:Whittaker-functors} identify standard Whittaker modules with a subspace of the completion of a Verma module. Therefore, we can restrict the above pairing to define a contravariant pairing between $(\overline{M(w \cdot \lambda)})_\eta=\overline{\Gamma}_\eta(M(w\cdot \lambda))$ and $M(w \cdot \lambda)$. By Proposition \ref{prop:Whittaker-functor}, $\overline{\Gamma}_\eta(M(w\cdot \lambda))=M(\lambda, \eta)$ when $w\in W_\eta$. Therefore, for each $w\in W_\eta$, this construction yields a contravariant pairing between $M(\lambda, \eta)$ and $M(w \cdot \lambda)$. However (in the setting where $\lambda+\rho$ is regular and integral), unless $w$ is the longest element of $W_\eta$, the pairing constructed in this way is trivial (i.e. the zero pairing). Therefore, this construction yields exactly one of the contravariant pairings $\langle \cdot, \cdot \rangle_w$ of Theorem \ref{thm:unique-forms}. It is interesting that the other contravariant pairings are not obtained in this way. 
\end{remark}

%% file: sec-costandard-modules.tex
\section{Costandard modules}
\label{costandard whittaker modules}

In this section we define costandard modules in the category $\N$ (Definition \ref{costandard whittaker modules}) and show that each contravariant pairing induces a $\g$-morphism from standard to costandard modules (Lemma \ref{lemma:costandard-objects}). These $\mf{g}$-morphisms are the algebraic analogues to the canonical maps between standard and costandard twisted Harish-Chandra sheaves introduced in Section \ref{sec:standard-and-simple-HC-sheaves}. Our costandard modules share many of the fundamental properties of dual Verma modules: each has a unique irreducible submodule and the same set of composition factors as the corresponding standard module (Theorem \ref{thm:alg-def-of-costandard-satisfies-universal-properties}). In fact, these two conditions provide a set of universal properties for costandard modules (Theorem \ref{thm:universal-properties}). %Using these results, we will show in Section \ref{whittaker modules form a highest weight category} that the $\g$-morphism from standard to costandard modules induced by a contravariant pairing is an algebraic analogue of the canonical map of corresponding $\mc{D}_\lambda$-modules obtained by Beilinson--Bernstein localization. 

Given a $U(\n)$-module $V$, recall from equation (\ref{eta twisted n-finite vectors}) that $(V)_\eta$ denotes the space of $\eta$-twisted $U(\n)$-finite vectors in $V$:
\begin{equation*}
    (V)_\eta := \{ v\in V:\text{$\forall u\in U(\n)$, $\exists$ $k$ s.t. }(u-\eta(u))^kv=0\}. 
\end{equation*}
\begin{lemma}[{\cite[Lem. 4.2.1]{Kostant}}]
\label{lem: twisted U(n)-finite vectors are g-stable}
For any $U(\g)$-module $V$, the subspace $(V)_\eta$ is a $U(\g)$-submodule.
\end{lemma}
\begin{proof}
Because the action map $\mf{g} \otimes_\C V \rightarrow V$ is a $\mf{g}$-module morphism, for any $X \in \mf{n}$, $Y \in \mf{g}$, $v \in (V)_\eta$, and $n \in \Z_{\geq 0}$ we have
\[
(X - \eta(X))^n \cdot Y \cdot v = \sum_{k=0}^n \bp n \\ k \ep (\ad X)^{n-k} Y \cdot (X - \eta(X))^k \cdot v.
\]
The adjoint action of $\mf{n}$ on $\mf{g}$ is nilpotent, so there exists $\ell \in \mathbb{Z}$ such that $(\ad X)^\ell Y = 0$. Because $v \in (V)_\eta$, there exists $m \in \mathbb{Z}$ such that $(X - \eta(X))^m \cdot v = 0$. Hence for any $X \in \mf{n}$, $Y \in \mf{g}$, $v \in (V)_\eta$, we can choose $n$ large enough so that every term in the sum is zero. For such an $n$, $(X - \eta(X))^n \cdot Y \cdot v = 0$, and thus $Y \cdot v \in (V)_\eta$. 
\end{proof}

%Suppose $x\in (X)_\eta$. Let $V = U(\n)x$ be the $U(\n)$-submodule generated by $x$, which is assumed to be finite-dimensional. The adjoint action of $\n$ on $\g$ is nilpotent and $\g = (\g)_0$. Equivalently, by Engel's theorem, $\n$ is a nilpotent Lie subalgebra of $\g$. It is well-known that because $\n$ is a nilpotent Lie subalgebra of $\g$ and $(V)_\eta$ is finite-dimensional, we have additivity of generalized weights: $(\g)_0\cdot (V)_\eta\subseteq (V)_{0+\eta}$. Because $x\in (V)_\eta$ and $(V)_\eta\subseteq (X)_\eta$, we have that $\g  \cdot x\subseteq (X)_\eta$. 

%Because the action map $\mf{g} \otimes_\C X \rightarrow X$ is a $\mf{g}$-module morphism, we have for any $X \in \mf{n}$, $Y \in \mf{g}$, $x \in (X)_\eta$, and $n \in \mathbb{Z}_{\geq 0}$
%\[
%(X - \eta(X))^n \cdot Y \cdot x = (X - \eta(X))^n \cdot Y \otimes x = (\ad X)^n Y \otimes x + Y \otimes (X - \eta(X))^n \cdot x. 
%\]
%The adjoint action of $\mf{n}$ on $\mf{g}$ is nilpotent, so there exists $k \in \Z$ such that $(\ad X)^k Y = 0$. Because $x \in (X)_\eta$, there exists $\ell \in \Z$ such that $(X - \eta(X))^\ell \cdot x= 0$. Hence
%\[
%(X - \eta(X))^{\ell + k} \cdot Y \cdot x = 0,
%\]
%so $Y \cdot x \in (X)_\eta$.

For a $U(\g)$-module $V$, denote by $V^* := \Hom_\C(V, \C)$ the full linear dual of $V$. The space $V^*$ becomes a $U(\g)$-module via the action $u \cdot f( - ) = f (\tau(u) \cdot - )$ for $u \in U(\g)$, $f \in V^*$.

\begin{definition}\label{def:costandard}
For $\lambda\in\h^\ast$, $\eta\in\ch\n$, and $w\in W_\eta$, define the $U(\g)$-module
\[
M^\vee_w(\lambda,\eta):=\left( M(w\cdot\lambda)^\ast\right)_\eta.
\]
We call $M^\vee_w(\lambda, \eta)$ the \emph{$w$-costandard module} corresponding to the standard module $M(\lambda,\eta)$.  
\end{definition}
We will later show, in Corollary \ref{cor:costandard-independence-of-w}, that $M^\vee_w(\lambda,\eta)\cong M^\vee_y(\lambda,\eta)$ for each $w,y\in W_\eta$. Hence, up to isomorphism, the definition of $w$-costandard modules does not depend on the choice of $w\in W_\eta$. However, we will retain notation which reflects the choice of $w\in W_\eta$ until Corollary \ref{cor:costandard-independence-of-w}.

\begin{remark}
We can describe the construction in Definition \ref{def:costandard} in terms of coinduction, similarly to a dual Verma module. Indeed, $M^\vee_w(\lambda, \eta)$ is the $\mf{g}$-submodule of $\eta$-twisted $U(\mf{n})$-finite vectors in
\[
\Hom_\C(U(\mf{g}) \otimes_{U(\mf{b})} \C_\lambda, \C) = \Hom_{U(\mf{b})} (U(\mf{g}), \C_\lambda) = \mathrm{coind}_{U(\mf{b})}^{U(\mf{g})} \C_\lambda.
\]
\end{remark}

\begin{lemma}\label{lemma:costandard-objects}
For any $\lambda\in\h^\ast$, $\eta\in\ch\n$, and $w\in W_\eta$, the $w$-costandard module $M_w^\vee(\lambda,\eta)$ is an object in the category $\N(\xi(\lambda),\eta)$. Moreover, the contravariant pairing  $\langle\cdot,\cdot\rangle_w: M(\lambda,\eta)\times M(w \cdot \lambda) \rightarrow \C$ defined in (\ref{contravariant pairing}) induces a $\g$-morphism
\[
\varphi_w:M(\lambda,\eta)\rightarrow M^\vee_w(\lambda,\eta). 
\]
\end{lemma}
\begin{proof}
Recalling the definitions of Section~\ref{sec:Whittaker-functors}, the formal completion of the dual Verma module $M^\vee(w \cdot \lambda)$ is canonically isomorphic to the linear dual of a Verma module (see \cite[\S3]{Backelin}):
\begin{align*}
\overline{M^\vee(w\cdot\lambda)} &\cong ((M^\vee(w\cdot\lambda)^\ast)_{\eta=0})^\ast \\
&\cong M(w\cdot\lambda)^\ast.
\end{align*}
Therefore, the Whittaker functor $\overline{\Gamma}$ applied to a dual Verma module is isomorphic to a costandard module:
\begin{align*}
\overline{\Gamma}_\eta(M^\vee(w\cdot\lambda))&= (M(w\cdot\lambda)^\ast)_\eta\\
&=M^\vee_w(\lambda,\eta). 
\end{align*}
Because $\overline{\Gamma}_\eta$ is a functor from $\cO_\lambda$ to $\N(\xi(\lambda),\eta)$, the $w$-costandard module $M_w^\vee(\lambda,\eta)$ is an object in $\N(\xi(\lambda),\eta)$. 

Each contravariant pairing $\langle \cdot, \cdot \rangle_w$ induces a $\g$-morphism
\[
\varphi_w:M(\lambda,\eta)\rightarrow M(w\cdot\lambda)^\ast, \hspace{2mm} v \mapsto \langle v, \cdot \rangle_w.
\]
Because $\varphi$ is a $\g$-morphism, the image of $\varphi$ is contained in $(M(w\cdot\lambda)^\ast)_\eta$. Therefore each contravariant pairing $\langle\cdot,\cdot\rangle_w$ defines a $\g$-morphism 
\[
\varphi_w:M(\lambda,\eta)\rightarrow M_w^\vee(\lambda, \eta), \hspace{2mm} v \mapsto \langle v, \cdot \rangle_w.
\]%(which by an abuse of notation we will also denote by $\langle\cdot,\cdot\rangle_w$) 
%\begin{align*}
%\langle\cdot,\cdot\rangle_w:M(\lambda,\eta)&\rightarrow M_w^\vee(\lambda,\eta)\\
%v&\mapsto \langle v ,\cdot\rangle_w.
%\end{align*}
\end{proof}
\begin{theorem}\label{thm:alg-def-of-costandard-satisfies-universal-properties}
Let $\lambda\in\h^\ast$, $\eta\in\ch\n$, and $w\in W_\eta$. Then
\begin{enumerate}
    \item $[M(\lambda,\eta)]=[M^\vee_w(\lambda,\eta)]$ in $\mathcal{K}\N(\xi(\lambda),\eta)$, and
    \item $M_w^\vee(\lambda,\eta)$ contains a unique irreducible submodule, which is isomorphic to $L(\lambda,\eta)$. 
\end{enumerate}
\end{theorem}
\begin{proof} 
In the proof of Lemma \ref{lemma:costandard-objects} it was shown that 
\begin{align*}
\overline{\Gamma}_\eta(M^\vee(w\cdot\lambda))&
=M^\vee_w(\lambda,\eta). 
\end{align*}
Because $\overline{\Gamma}_\eta$ is exact, we have a homomorphism of Grothendieck groups: $\overline{\Gamma}_\eta:\mathcal{K}\cO\rightarrow\mathcal{K}\N$. Because $M(\lambda,\eta)=\overline{\Gamma}_\eta(M(w\cdot\lambda))$ \cite[Prop. 6.9]{Backelin} and $[M(w\cdot\lambda)] = [M^\vee(w\cdot\lambda)]$ in $\mathcal{K}\cO$, we conclude that 
\begin{align*}
[M(\lambda,\eta)]
& = \overline{\Gamma}_\eta ([M(w\cdot\lambda)]) \\
&= \overline{\Gamma}_\eta([M^\vee(w\cdot\lambda)])\\
&= [M^\vee_w(\lambda,\eta)]. 
\end{align*}
This proves (1). 

Again, by Lemma \ref{lemma:costandard-objects}, each contravariant pairing $\langle\cdot,\cdot\rangle_w$ defines a $\g$-morphism 
\[
\varphi_w:M(\lambda,\eta)\rightarrow M_w^\vee(\lambda,\eta), \hspace{2mm} v\mapsto \langle v ,\cdot\rangle_w.
\]
The kernel of this morphism is the left radical of the form, $\text{Rad}^L\langle \cdot, \cdot \rangle_w$. By Proposition \ref{prop:radicals}, $\text{Rad}^L\langle \cdot, \cdot \rangle_w$ is the unique maximal proper submodule of $M(\lambda,\eta)$. Hence the image of $\varphi_w$ is isomorphic to $L(\lambda,\eta)$. 

We will now show that $\im\varphi_w$ is the unique irreducible submodule of $M^\vee_w(\lambda,\eta)$. Because $M(w\cdot\lambda)$ is a free rank 1 left $U(\bar{\n})$-module, we have that
\begin{align*}
H^0_\eta(\n,M_w^\vee(\lambda,\eta))& = H^0_\eta(\n, M(w\cdot\lambda)^\ast)\\
& \cong H_0(\bar{\n},M(w\cdot\lambda))^\ast\\
& \cong \C_\lambda,
\end{align*}
where the isomorphism are as $\z$-modules. 

Suppose that $X$ is an irreducible submodule of $M^\vee_w(\lambda,\eta)$. Then 
\[
H^0_\eta(\n,X)\subset H^0_\eta(\n,M^\vee_w(\lambda,\eta))\cong \C.
\]
Any irreducible object in $\mc{N}$ must contain a Whittaker vector (Proposition \ref{whittaker facts}(v)), so $H_\eta^0(\n, X) \neq 0$. Therefore, $H^0_\eta(\n,X)= H^0_\eta(\n,M^\vee_w(\lambda,\eta))$ and $X\cap \im\langle \cdot, \cdot \rangle_w\neq 0$. By irreducibility, $X = \im\langle \cdot, \cdot \rangle_w$. 
\end{proof}

The remaining results of this section, concluding with Theorem \ref{thm:universal-properties}, show that costandard Whittaker modules satisfy a universal property. 

\begin{lemma}\label{lemma:nonzero-homology-implies-gmorphism}
Assume $\lambda+\rho\in\h^\ast$ is regular. Let $V\in\N(\xi(\lambda),\eta)$ be a module such that 
\[H_0(\bar{\n},V)_{w\cdot\lambda}\neq 0,\]
for some $w\in W_\eta$. Then there exists a nonzero $\g$-morphism from $V$ to $M_w^\vee(\lambda,\eta)$.
\end{lemma}
\begin{proof}
We begin the proof by showing that $\dim H_0(\bar{\n},V)<\infty$. By Theorem \ref{thm:homology}, $H_0(\bar{\n},M(\mu,\eta))$ is finite-dimensional for each $\mu\in W_\eta\cdot\lambda$. Because each simple module $L(\mu,\eta)\in\N(\xi(\lambda),\eta)$ is the quotient of a standard module and $\C\otimes_{\bar{\n}}- = H_0(\bar{\n},-)$ is right exact,  $H_0(\bar{\n},L(\mu,\eta))$ is finite dimensional for each simple module $L(\mu,\eta)\in\N(\xi(\lambda),\eta)$. Because $V$ is finite length, $H_0(\bar{\n},V)$ is finite dimensional. 

Therefore, if $H_0(\bar{\n},V)_{w\cdot\lambda}\neq 0$, then $(H_0(\bar{\n},V)^\ast)_{w\cdot\lambda}\neq 0$. 
By tensor-hom adjuction, 
\[
H_0(\bar{\n},V)^\ast\cong \Hom_\C(\C\otimes_{\bar{\n}}V,\C) \cong \Hom_\n(\C,\Hom_\C(V,\C))\cong H^0(\n,V^\ast).
\]
We have shown that if $H_0(\bar{\n},V)_{w\cdot\lambda}\neq 0$, then $H^0(\n,V^\ast)_{w\cdot\lambda}\neq 0$. Therefore, there exists a nonzero $\g$-morphism 
\[
\varphi:M(w\cdot\lambda)\rightarrow V^\ast. 
\]
The morphism $\varphi$ determines a $\g$-morphism
\begin{align*}
    \hat{\varphi}:V&\rightarrow M(w\cdot\lambda)^\ast
\end{align*}
given by $\hat{\varphi}(v)(x):=\varphi(x)(v)$ for  $v\in V$ and $x\in M(w\cdot\lambda)$. We confirm that $\hat{\varphi}$ is a $\g$-morphism with the following simple calculation. For $u\in\g$, $v\in V$, and $x\in M(w\cdot\lambda)$, we have
\begin{align*}
    \hat{\varphi}(uv)(x) &= \varphi(x)(uv)\\
    &= (\tau(u)\varphi(x))(v)\\
    &= \varphi(\tau(u)x)(v)\\
    &= \hat{\varphi}(v)(\tau(u)x)\\
    &= (u\hat{\varphi}(v))(x).
\end{align*}
Here $\tau:U(\g) \rightarrow U(\g)$ is the transpose antiautomorphism (Section \ref{contravariant pairings of standard whittaker modules}). Because $\varphi$ is nonzero, there exists $x\in M(w\cdot\lambda)$ and $v\in V$ such that $\varphi(x)(v)\neq 0$. Therefore, $\hat{\varphi}(v)\neq 0$ and $\hat{\varphi}$ is nonzero. Because $V\in \N(\xi(\lambda),\eta)$, the image of $\hat{\varphi}$ must be contained in $M(w\cdot\lambda)^\ast_\eta =  M_w^\vee(\lambda,\eta)$.  
\end{proof}

The following proposition involves categories of $\lev_\eta$-modules as well as categories of $\g$-modules, and will require some additional notation. For $\lambda\in\h^\ast$ and $\eta\in\ch\n$, let $\N_\g \left(\widehat{\xi(\lambda)},\eta\right)$ be the category of $\g$-modules which are finitely generated, locally $U(\n)$-finite, locally $Z(\g)$-finite, annihilated by a power of $\xi(\lambda)$, and locally annihilated by a power of $\ker\eta$ (see Proposition \ref{prop:Ndefinition}). Recall that, for $\mu\in\h^\ast$, $\xi_\eta(\mu)$ is a maximal ideal in $Z(\lev_\eta)$ (see Section \ref{sec:standard-whittaker-modules}). Let $\N_{\lev_\eta}\left(\widehat{\xi_\eta(\mu)},\eta\right) $ be the category of $\lev_\eta$-modules which are finitely generated, locally $U(\n_\eta)$-finite, locally $Z(\lev_\eta)$-finite, annihilated by a power of $\xi_\eta(\mu)$, and locally annihilated by a power of $\ker\eta\vert_{U(\n_\eta)}$. For a module $X\in \N_\g \left(\widehat{\xi(\lambda)},\eta\right)$, let $(X)_{\mu_\z}$ denote the generalized $\z$-weight space corresponding to the weight $\mu\in\h^\ast\subset \z^\ast$ (where $\z$ is the center of $\lev_\eta$, see Proposition \ref{whittaker facts}). 

\begin{proposition}[{\cite[Prop. 2.2.2.]{Brown2019}}]\label{prop:exact-weight-space-functor}

The projection from a module $X\in \N_\g \left(\widehat{\xi(\lambda)},\eta\right)$ onto a generalized $\z$-weight space $(X)_{\mu_\z}$ defines an exact functor
\begin{align*}
    (\cdot)_{\mu_\z}:\N_\g \left(\widehat{\xi(\lambda)},\eta\right)\rightarrow\N_{\lev_\eta}\left(\widehat{\xi_\eta(\mu)},\eta\right).
\end{align*}
\end{proposition}

% I wrote the following Lemma, but I think we won't need it. 
% \begin{lemma}\label{lemma:nonzero-homology}
% If $\eta$ is nondegenerate and $0\neq V\in \N(\hat{\chi}_\lambda,\eta)$, then $H_0(\bar{\n},V)_\lambda\neq 0$.
% \end{lemma}
% \begin{proof}
% By \cite[Lemma 5.8]{TwistedSheaves}, the functor 
% \[
% H^0_\eta(\n,\cdot):\N\left(\widehat{\chi_\lambda},\eta\right)\rightarrow Z(\g)\text{-mod}
% \]
% is exact. Additionally, $\dim H^0_\eta(\n, Y)=1$ for each simple module $Y\in \N\left(\widehat{\chi_\lambda},\eta\right)$. Therefore, $\dim H^0_\eta(\n,V)=\text{ length }V$. This implies that if $V\neq 0$, then $H^0_\eta(\n,V)\neq 0$. Let $v$ be a nonzero vector in $H^0_\eta(\n,V)$. By \cite[Theorem 5.9]{TwistedSheaves}, 
% \[V\cong U(\g)\otimes_{Z(\g)\otimes U(\n)}H^0_\eta(\n,V).\]
% To shorten notation in the remainder of the proof, let $W$ to denote $U(\g)\otimes_{Z(\g)\otimes U(\n)}H^0_\eta(\n,V)$. 

% Because $U(\g)$ is a free $(U(\bar{\n}),Z(\g)\otimes U(\n))$-bimodule by Lemma \ref{U-free}, we have that $1\otimes v \in W$ is not in the $\bar{\n}$-span of any vector in $W$. Therefore, \[
% 0\neq 1\otimes v \in \C\otimes_{\bar{\n}} W := H_0(\bar{\n},W) \cong H_0(\bar{\n},V). 
% \]
% \end{proof}

\begin{lemma}\label{lemma:char-implies-nonzero-homology}
Assume $\lambda+\rho\in\h^\ast$ is regular. Let $V\in\N(\xi(\lambda),\eta)$. If $[V]=[M(\lambda,\eta)]$ in the Grothendieck group $\mathcal{K}\N$, then $H_0(\bar{\n},V)_{w\cdot\lambda}\neq 0$ for each $w\in W_\eta$. 
\end{lemma}
\begin{proof}
By Proposition~\ref{whittaker facts}, $M(\lambda,\eta)_{\mu_\z}=0$ for each $\mu>\lambda$, where $>$ denotes the partial order on $\z^\ast$ induced by $\Sigma^+-\Sigma^+_\eta$ (see \cite[\S1]{McDowell}). Hence by Proposition~\ref{prop:exact-weight-space-functor}, $V_{\mu_\z}=0$ for each $\mu>\lambda$. 
By \cite[Prop. 5]{McDowell},
\[
\bar{\n}^\eta V_{\mu_\z}\subset \bigoplus_{\nu<\mu} V_{\nu_\z}.
\]
Therefore, $(\bar{\n}^\eta V)_{\lambda_\z}=0$. (If not, then there must exist $\mu>\lambda$ such that $V_{\mu_\z}\neq 0$, a contradiction). This implies that $ (\bar{\n}V)_{\lambda_\z}\subset (\bar{\n}_\eta V)_{\lambda_\z}$ because projection onto a generalized weight space is linear and $\bar{\n}V= \bar{\n}^\eta V+\bar{\n}_\eta V$. Moreover, because $\z$ commutes with $\bar{\n}_\eta$, the action of $\bar{\n}_\eta$ preserves $\z$-weights, and $(\bar{\n}_\eta V)_{\lambda_\z} = \bar{\n}_\eta V_{\lambda_\z}$. This shows that if $v\in V_{\lambda_\z}$ and $v\not\in \bar{\n}_\eta V_{\lambda_\z}$ then $v\not\in\bar{\n} V$. In other words, the map $H_0(\bar{\n}_\eta, V_{\lambda_\z})\rightarrow H_0(\bar{\n},V)$ induced by inclusion is injective. 

By Proposition \ref{whittaker facts}, $M(\lambda,\eta)_{\lambda_\z}\cong Y_{\lev_\eta}(\lambda,\eta)$ as $\lev_\eta$-modules. Because $[V]=[M(\lambda,\eta)]$, Proposition \ref{prop:exact-weight-space-functor} then implies that $V_{\lambda_\z}\cong Y_{\lev_\eta}(\lambda,\eta)$ as well. An application of Theorem \ref{thm:homology} to the $\lev_\eta$-module $Y_{\lev_\eta}(\lambda,\eta)$ lets us conclude that $H_0(\bar{\n}_\eta, V_{\lambda_\z})_{w\cdot\lambda}\neq 0$ for each $w\in W_\eta$. Therefore, $H_0(\bar{\n}, V)_{w\cdot\lambda}\neq 0$ for each $w\in W_\eta$ by the injectivity of $H_0(\bar{\n}_\eta, V_{\lambda_\z})\rightarrow H_0(\bar{\n},V)$.

\end{proof}
\begin{theorem}\label{thm:universal-properties}
Assume $\lambda+\rho\in\h^\ast$ is regular. Suppose $V\in\N$ is a module such that 
\begin{enumerate}
    \item $[V]=[M(\lambda,\eta)]$ in $\mathcal{K}\N$, and
    \item $V$ contains a unique irreducible submodule which is isomorphic to $L(\lambda,\eta)$. 
\end{enumerate}
Then $V\cong M_w^\vee(\lambda,\eta)$ for each $w\in W_\eta$. 
\end{theorem}
\begin{proof}

By Lemma \ref{lemma:char-implies-nonzero-homology}, $[V]=[M(\lambda,\eta)]$ implies that $H_0(\bar{\n},V)_{w\cdot\lambda}\neq 0$ for each $w\in W_\eta$. Therefore, we can apply Lemma \ref{lemma:nonzero-homology-implies-gmorphism} to get a nonzero $\g$-morphism $\phi_w:V\rightarrow M_w^\vee(\lambda,\eta)$ for each $w\in W_\eta$.

Recall that, by Theorem \ref{thm:alg-def-of-costandard-satisfies-universal-properties}, $M^\vee_w(\lambda,\eta)$ contains a unique irreducible submodule which is isomorphic to $L(\lambda,\eta)$. Because $\phi_w$ is nonzero, $\im\phi_w$ must contain the unique irreducible submodule of $M_w^\vee(\lambda,\eta)$. 

Let $K=\ker\phi_w$. The assumption that $[V]=[M(\lambda,\eta)] $ implies (by Proposition \ref{whittaker facts} and Proposition \ref{prop:exact-weight-space-functor}) that $M_w^\vee(\lambda,\eta)_{\lambda_\z}$ and $V_{\lambda_\z}$ are irreducible $\lev_\eta$-modules. 
Let $(\phi_w)_{\lambda_\z}:V_{\lambda_\z}\rightarrow M_w^\vee(\lambda,\eta)_{\lambda_\z}$ be the restriction of $\phi_w$ to the generalized $\z$-weight space of weight $\lambda$. By irreducibility of the $\lev_\eta$-modules and Dixmier's lemma, either $(\phi_w)_{\lambda_\z}=0$ or $(\phi_w)_{\lambda_\z}$ is an isomorphism. Because the unique irreducible $\g$-submodule of $M_w^\vee(\lambda,\eta)$ contains the weight space $M_w^\vee(\lambda,\eta)_{\lambda_\z}$ and $\im\phi_w$ contains the unique irreducible $\g$-submodule, we conclude that $M^\vee_w(\lambda,\eta)_{\lambda_\z}\subset \im\phi_w$. Therefore, $(\phi_w)_{\lambda_\z}$ is an isomorphism and $K_{\lambda_\z} = 0$. 

Because $\phi_w$ is a $\g$-morphism, $K$ is a $\g$-submodule of $V$. Moreover, $K$ is finite length and must have an irreducible submodule $I\subset K$. Because $K_{\lambda_\z} = 0$, $I_{\lambda_\z} = 0$, hence $I\not\cong L(\lambda,\eta)$. Therefore $K=0$ by uniqueness of irreducible $\g$-submodules of $V$. 

Because $K=0$, the map $\phi$ is injective. The assumption that $ [V]=[M(\lambda,\eta)] $ implies that $\phi_w$ is surjective, hence an isomorphism. 

\end{proof}

Theorem \ref{thm:universal-properties} immediately implies that all $w$-costandard modules corresponding to a standard module $M(\lambda, \eta)$ are isomorphic.

\begin{corollary}\label{cor:costandard-independence-of-w}
Assume $\lambda+\rho \in \h^*$ is regular. Then 
\[
M^\vee_w(\lambda,\eta)\cong M^\vee_y(\lambda,\eta)
\]
for all $w,y\in W_\eta$. 
\end{corollary}

\begin{remark} Now that we have determined that $M_w^\vee(\lambda,\eta)\cong M_y^\vee(\lambda,\eta)$ for any $w,y\in W_\eta$, we will omit the subscript and refer to the $\g$-module $M^\vee(\lambda,\eta)$ as a \emph{costandard module}.\end{remark}

%% file: sec-highest-weight-category.tex
\section{Whittaker modules form a highest weight category}
\label{whittaker modules form a highest weight category}

In this section we prove that the category $\mc{N}(\xi(\lambda), \eta)$ is a highest weight category. To do so we work in the geometric category $\mc{M}_{\text{coh}}(\mc{D}_\lambda, N, \eta)$. 

\begin{definition} Fix a field $k$ and let $\mc{A}$ be a $k$-linear category. We say that $\mc{A}$ is a {\em highest weight category}\footnote{This definition aligns with \cite[\S3.2]{BGS}, and differs slightly from the original definition of highest weight categories in \cite{CPS}.} if there exists a finite poset $\Lambda$ so that $\mc{A}$ and $\Lambda$ satisfy the following conditions:
\begin{enumerate}[label=(\arabic*)]
\item $\mc{A}$ is finite-length. 
\item The set of simple objects (up to equivalence) in $\mc{A}$ is finite and parameterized by $\Lambda$. Denote by $L_\lambda \in \mc{A}$ the simple object corresponding to $\lambda \in \Lambda$. 
\item For each $\lambda \in \Lambda$, there exists a standard object $M_\lambda$ and costandard object $I_\lambda$ in $\mc{A}$ and morphisms $M_\lambda \rightarrow L_\lambda$ and $L_\lambda \rightarrow I_\lambda$.
\item For any simple object $L_\lambda \in \mc{A}$, $\End (L_\lambda) = k$. 
\item If $T \subseteq \Lambda$ is closed (i.e. if $\mu \leq \lambda$ and $\lambda \in T$, then $\mu \in T$) and $\lambda \in T$ is maximal, $M_\lambda \rightarrow L_\lambda$ (resp. $L_\lambda \rightarrow I_\lambda$) is a projective cover\footnote{A projective cover of an object $M$ in a category $\mc{C}$ is a morphism $P \xrightarrow{f} M$ out of a projective object $P \in \mc{C}$ which is a superfluous epimorphism, meaning that every morphism $N \xrightarrow{g} P$ with the property that $f \circ g$ is an epimorphism is itself an epimorphism.} (resp. injective hull) in the Serre subcategory $\mc{A}_T \subset \mc{A}$ generated by the simple objects $L_\lambda$ for $\lambda \in T$. 
\item For $\lambda \in \Lambda$, the kernel of $M_\lambda \rightarrow L_\lambda$ is in $\mc{A}_{<\lambda}$, as is the cokernel of $L_\lambda \rightarrow I_\lambda$.
\item For all $\lambda, \mu \in \Lambda$, $\Ext^2_\mc{A}(M_\lambda, I_\mu) = 0$. 
\end{enumerate}
\end{definition}

\begin{theorem}\label{thm:geometric-highest-weight}
The category $\mc{M}_{\text{coh}}(\mc{D}_\lambda, N, \eta)$ is a highest weight category. 
\end{theorem}
\begin{proof}
Denote the category $\mc{M}_{\text{coh}}(\mc{D}_\lambda, N, \eta)$ by $\mc{A}$. The Bruhat order on longest coset representatives defines a partial order on the finite set $W_\eta \backslash W$. We will show that the pair $(\mc{A}, W_\eta \backslash W)$ satisfies the seven conditions of a highest weight category. 

The category $\mc{A}$ is a finite-length abelian category, so condition (1) is satisfied. Let $\mc{M}_C:=\mc{M}(w^C, \lambda, \eta)$,  $\mc{I}_C:=\mc{I}(w^C, \lambda, \eta)$ and $\mc{L}_C:=\mc{L}(w^C, \lambda, \eta)$ be the standard, costandard, and simple $\eta$-twisted Harish-Chandra sheaves in $\mc{A}$ (see Section \ref{sec:standard-and-simple-HC-sheaves}). These are parametrized by the poset $W_\eta \backslash W$. The $\mc{D}_\lambda$-module $\mc{L}_C$ appears as the unique irreducible quotient of $\mc{M}_C$ and the unique irreducible subsheaf of $\mc{I}_C$ \cite[\S3]{TwistedSheaves} \cite[Prop. 3]{Romanov}, so there are projection and inclusion maps 
\[
\mc{M}_C \twoheadrightarrow \mc{L}_C  \text{ and } \mc{L}_C \hookrightarrow I_C.
\]
Moreover, all simple objects in $\mc{A}$ are of the form $\mc{L}_C$ for some $C \in W_\eta \backslash W$ \cite[\S3]{TwistedSheaves}. Hence conditions (2) and (3) are satisfied. 

By Schur's lemma, $\End(\mc{L}_C)$ is a division algebra over $\C$. Restriction to $C(w^C)$ gives a nonzero algebra homomorphism 
\[
\varphi: \End(\mc{L}_C) \rightarrow \End(\mc{O}_{C(w^C), \eta})=\C.
\]
Since $\End(\mc{L}_C)$ is a division algebra, $\varphi$ must be an isomorphism. (Indeed, the kernel of $\varphi$ is an ideal in $\End(\mc{L}_C)$, so it must be trivial, and $\varphi$ is nonzero, so it must be surjective.) This establishes (4). 

Next we argue (5). For a fixed coset $C\in W_\eta\backslash W$, let 
\[
T= \{ D \in W_\eta \backslash W \mid D \leq C \}.
\]
Because the Bruhat order on longest coset representatives agrees with the closure order on Bruhat cells \cite[Ch.6 \S1]{Localization}, $\mc{A}_T$ is the category of $\eta$-twisted Harish-Chandra sheaves supported on $\overline{C(w^C)}$. To establish the projectivity of $\mc{M}_C$, we will show that the functor 
\[
\Hom_{\mc{A}_T}(\mc{M}_C, -): \mc{A}_T \rightarrow \Vect
\]
is exact. Let $j_{w^C}:C(w^C) \hookrightarrow \overline{C(w^C)}$ be inclusion. This is an open immersion, so the functor $j_{w^C}^!=j_{w^C}^+$ is the restriction functor (see \cite[App. A.2]{Romanov} for conventions on $\mc{D}$-module functors). In particular, $j_{w^C}^!$ is exact. Denote by $\mc{A}_C$ the category of $\eta$-twisted Harish-Chandra sheaves on $C(w^C)$. For any $\mc{V} \in \mc{A}_T$, 
\[
\Hom_{\mc{A}_T}(j_{w^C!}\mc{O}_{C(w^C),\eta}, \mc{V}) = \Hom_{\mc{A}_C}(\mc{O}_{C(w^C), \eta}, j_{w^C}^! \mc{V}).
\]
The category $\mc{A}_C$ is semisimple, so there are no higher extensions. Hence the composition $\Hom(\mc{O}_{C(w^C)}, -) \circ j_{w^C}^!$ is exact. This establishes that $\mc{M}_C$ is projective. 

The category $\mc{A}$ is finite length and abelian, so it is Krull-Schmidt\footnote{Recall that a category is called {\em Krull-Schmidt} if every object decomposes into a finite direct sum of indecomposable objects, which are characterized by the fact that their endomorphism rings are local.}. In a Krull-Schmidt category, any indecomposable projective object which surjects onto a given object is a projective cover, so to show that $\mc{M}_C \twoheadrightarrow \mc{L}_C$ is a projective cover, it suffices to show that $\mc{M}_C$ is indecomposable. We will do so by showing that its endomorphism ring is local. Indeed, we compute:
\begin{align*}
    \End(\mc{M}_C) &= \Hom(j_{w^C!} \mc{O}_{C(w^C), \eta}, j_{w^C!} \mc{O}_{C(w^C), \eta})\\
    &= \Hom(\mc{O}_{C(w^C), \eta}, j_{w^C}!j_{w^C!} \mc{O}_{C(w^C), \eta}) \\
%    &= \Hom(\mc{O}_{C(w^C), \eta}, j_{w^C}^+ j_{w^C!} \mc{O}_{C(w^C), \eta}) \\
    &= \Hom(\mc{O}_{C(w^C), \eta}, \mc{O}_{C(w^C), \eta}) \\ 
    &= \C. 
\end{align*}
This shows that $\mc{M}_C \twoheadrightarrow \mc{L}_C$ is a projective cover in $\mc{A}_T$. By applying holonomic duality, we obtain that $\mc{L}_C\hookrightarrow \mc{I}_C$ is an injective hull in $\mc{A}_T$, establishing (5).

To etablish (6), note that $\mc{M}_C$, $\mc{I}_C$, and $\mc{L}_C$ all restrict to the same object on the biggest cell in their support, and the natural maps between them restrict to isomorphisms on this cell. Hence the support of the kernel and cokernel is strictly smaller.

%Property (6) follows immediately from the fact that composition multiplicities of $\mc{M}_C$ and $\mc{I}_C$ are determined by parabolic Kazhdan--Lusztig polynomials \cite[Cor. 4, Rem. 4]{Romanov}.

It remains to show (7). 
%For $w \in W$, denote by $\mc{D}_{w, \lambda}$ the twisted sheaf of differential operators on $C(w)$ obtained by pulling back $\mc{D}_\lambda$ via the inclusion $i_w:C(w) \rightarrow X$ \cite[Ch.1, \S1]{Localization}. Let $D^b(\mc{D}_\lambda)$ be the bounded derived category of $\mc{D}_\lambda$-modules on $X$ and $D^b(\mc{D}_{w,\lambda})$ the bounded derived category of $\mc{D}_{w,\lambda}$-modules on $C(w)$. 
Let $C, D \in W_\eta \backslash W$ be cosets. We have
\begin{align*}
    \Ext^2_{\mc{M}}(\mc{M}_C, \mc{I}_D) &= \Hom_{D^b(\mc{A})}(i_{w^C!}\mc{O}_{C(w^C), \eta}, i_{w^D+}\mc{O}_{C(w^D),\eta}[2]) \\
    &= \Hom_{D^b(\mc{A}_C)}(\mc{O}_{C(w^C), \eta}, i_{w^C}^!i_{w^D+}\mc{O}_{C(w^D), \eta}[2]).
\end{align*}
By smooth base change \cite[Thm. 10.2]{D-modulesnotes} applied to the fibre product diagram 
\[
\begin{tikzcd}
C(w^C) \times_X C(w^D) \arrow[r] \arrow[d] & C(w^C) \arrow[d, hookrightarrow, "i_{w^C}"] \\
C(w^D) \arrow[r, hookrightarrow, "i_{w^D}"] & X 
\end{tikzcd}
\]
we have that $i_{w^C}^! i_{D+} \mc{O}_{C(w^C), \eta} = 0$ if $C \neq D$. 

If $C=D$, then $i^!_{w^C}i_{w^C+}\mc{O}_{C(w^C), \eta} \cong \mc{O}_{C(w^C), \eta}$ (viewed as a complex concentrated in degree $0$ in $D^b(\mc{A}_C)$). Since the category $\mc{A}_C$ is semisimple, there are no higher extensions, hence 
\[
\Hom_{D^b(\mc{A}_C)}(\mc{O}_{C(w^C), \eta}, i_{w^C}^!i_{w^C+}\mc{O}_{C(w^C), \eta}[2]) = \Ext^2_{\mc{A}_C}(\mc{O}_{C(w^C), \eta}, \mc{O}_{C(w^C), \eta}) = 0.
\]
\end{proof}

We wish to use Theorem \ref{thm:geometric-highest-weight} to give $\mc{N}(\xi(\lambda), \eta)$ the structure of a highest weight category. It was established in \cite[ Thm. 9, Thm. 10]{Romanov} that the global sections of standard (resp. simple) $\eta$-twisted Harish-Chandra sheaves are standard (resp. simple) Whittaker modules (see Proposition \ref{prop:global-sections-of-standard}). Moreover, using the universal property of costandard modules established in Section \ref{costandard whittaker modules} (Theorem \ref{thm:universal-properties}), we can show that the global sections of costandard $\eta$-twisted Harish-Chandra sheaves are the costandard modules defined in Section \ref{costandard whittaker modules}. We do so in the following lemma.

\begin{lemma}\label{lemma:algebraic-geometric-agreement}
Let $\lambda+\rho\in\h^\ast$ be regular and antidominant. For each $C \in W_\eta \backslash W$, 
\[
\Gamma(X,\mc{I}(w^C,\lambda+\rho,\eta)) \cong M^\vee(w^C\cdot\lambda,\eta). 
\]
\end{lemma}
\begin{proof}
By Theorem \ref{thm:universal-properties}, it is enough to show that:
\begin{enumerate}
    \item \label{proof:condition1} $[\Gamma(X,\mc{I}(w^C,\lambda+\rho,\eta))] = [M(w^C\cdot\lambda,\eta)]$, and
    \item\label{proof:condition2}  $\Gamma(X,\mc{I}(w^C,\lambda+\rho,\eta))$ contains a unique irreducible submodule which is isomorphic to $L(\lambda,\eta)$. 
\end{enumerate}
We begin with a proof of (\ref{proof:condition1}). Let $D^b(\mc{M}_{qc}(\mc{D}_\mu))$ denote the bounded derived category of quasi-coherent $\mc{D}_\mu$-modules on $X$, and $w \mu$ the regular action of $W$ on $\h^*$ (not the dot action). For $w \in W$ and $\mu \in \h^*$, let
\[
LI_w:D^b(\mc{M}_{qc}(\mc{D}_\mu)) \rightarrow D^b(\mc{M}_{qc}(\mc{D}_{w\mu}))
\]
be the corresponding intertwining functor (see \cite[Ch. 3 \S3]{Localization} for a definition of $LI_w$). Let $w_C \in W$ denote the unique shortest element in a coset $C \in W_\eta \backslash W$. (Recall that we denote the longest element of $C$ by $w^C$.) Denote the $W_\eta$-coset of the identity $1\in W$ by $\Theta$. Then the longest element in $\Theta = W_\eta$ is $w^\Theta$. For every coset $C \in W_\eta \backslash W$, we have $w^C=w^\Theta  w_C$ \cite[Ch. 6 Thm. 1.4(iv)]{Localization}. By \cite[Prop. 5]{Romanov}, for any $\mu \in \h^*$, 
\[
LI_{w^{-1}_C}(\mc{I}(w^\Theta , \mu + \rho, \eta)) = \mc{I}(w^C, w_{C}^{-1}(\mu + \rho), \eta). 
\]
Hence 
\[
R\Gamma\circ LI_{w^{-1}_C}(\mc{I}(w^\Theta , \mu + \rho, \eta)) = R \Gamma (\mc{I}(w^C, w_{C}^{-1}(\mu + \rho), \eta)).
\]
If we choose $\mu$ such that $w_C^{-1}(\mu + \rho)$ is antidominant, then by \cite[Ch. 3, Thm. 3.23]{Localization}, we have
\[
R\Gamma( \mc{I}(w^\Theta , \mu + \rho, \eta)) = R \Gamma (\mc{I}(w^C, w_C^{-1}(\mu + \rho), \eta)).
\]
It was shown in the proof of \cite[Prop. 5]{Romanov} that for any $\mu \in \h^*$, 
\[
[\Gamma(X, \mc{I}(w^\Theta , \mu + \rho, \eta))] = [M(w^\Theta  \cdot \mu, \eta)]
\]
in $K\mc{N}(\xi(\mu), \eta)$. Hence for $\lambda$ such that $\lambda + \rho$ is antidominant, 
\begin{align*}
    [\Gamma(X, \mc{I}(w^C, \lambda + \rho, \eta)]&=[\Gamma(X, \mc{I}(w^\Theta , w_C(\lambda + \rho), \eta)]\\
    &=[M(w^\Theta  w_C \cdot \mu, \eta)]\\
    &= [M(w^C \cdot \mu, \eta)],
\end{align*}
which completes the proof of (\ref{proof:condition1}). 

To prove (\ref{proof:condition2}), we recall that $\mc{I}(w^C, \lambda + \rho, \eta)$ contains a unique irreducible subsheaf, $\mc{L}(w^C, \lambda + \rho, \eta)$. By Proposition \ref{prop:global-sections-of-standard}, when $\lambda + \rho$ is regular and antidominant, $\Gamma(X, \mc{L}(w^C, \lambda + \rho, \eta))=L(w^C \cdot \lambda, \eta)$. Hence by the exactness of $\Gamma$ for antidominant $\lambda + \rho$, $\Gamma(X, \mc{I}(w^C, \lambda + \rho, \eta))$ contains a unique irreducible submodule isomorphic to $L(w^C \cdot \lambda, \eta)$.
\end{proof}
\begin{corollary}\label{cor:highest-weight}
Let $\lambda+\rho$ be regular and antidominant. The category $\mc{N}(\xi(\lambda), \eta)$ is a highest weight category with standard objects $M(w^C\cdot \lambda, \eta)$, costandard objects $M^\vee(w^C\cdot \lambda, \eta)$, and simple objects $L(w^C \cdot\lambda, \eta)$, where $C$ ranges over all cosets in $W_\eta \backslash W$.  
\end{corollary}
\begin{proof}
 This follows immediately from the equivalence (\ref{BB equivalence}), Proposition \ref{prop:global-sections-of-standard}, Theorem \ref{thm:geometric-highest-weight}, and Lemma \ref{lemma:algebraic-geometric-agreement}. 
\end{proof}

%\begin{remark}
%Corollary \ref{cor:highest-weight} lets us conclude by \cite[Thm. 3.2.1]{BGS} that $\mc{N}(\xi(\lambda), \eta)$ has enough projective objects and enough injective objects. Moreover, each projective object has a finite filtration with standard subquotients and each injective object has a finite filtration with costandard subquotients. Corollary \ref{cor:highest-weight} also guarantees the existence of indecomposable tilting\footnote{An object in a highest weight category is called {\em tilting} if it has both a filtration with standard subquotients and a filtration with costandard subquotients.} objects in $\mc{N}(\xi(\lambda), \eta)$. 
%\end{remark}

%% file: sec-BGG-Reciprocity.tex
\section{BGG Reciprocity}
\label{sec:BGG-Reciprocity}
In this section we give an application of the previous results by generalizing the Bernstein--Gelfand--Gelfand reciprocity formulas to $\N$ (Theorem \ref{thm:bgg}). We begin by recalling some well-known properties of highest weight categories. By Corollary \ref{cor:highest-weight} and \cite[Thm. 3.2.1]{BGS}, we conclude that $\mc{N}(\xi(\lambda), \eta)$ has enough projective objects and enough injective objects. Moreover, each projective object has a finite filtration with standard subquotients (we refer to this filtration as the \emph{standard filtration}) and each injective object has a finite filtration with costandard subquotients. 
\begin{remark}
Corollary \ref{cor:highest-weight} also guarantees the existence of indecomposible tilting\footnote{An object in a highest weight category is called {\em tilting} if it has both a filtration with standard subquotients and a filtration with costandard subquotients.} objects in $\mc{N}(\xi(\lambda), \eta)$. 
\end{remark}

Let $P(\lambda,\eta)$ be a projective cover of $L(\lambda,\eta)$, and 
\begin{align*}
    \big(P(\lambda,\eta):M(\mu,\eta)\big)
\end{align*}
be the multiplicity of the standard Whittaker module $M(\mu,\eta)$ in the standard filtration of $P(\lambda,\eta)$. Let 
\begin{align*}
    \big[M^\vee(\mu,\eta):L(\lambda,\eta)\big]
\end{align*}
be the multiplicity of the irreducible module $L(\lambda,\eta)$ in the Jordan--H\"older filtration of $M^\vee(\mu,\eta)$. 
\begin{theorem}[BGG Reciprocity for $\N$]
\label{thm:bgg}
\begin{align*}
    \big(P(\lambda,\eta):M(\mu,\eta)\big)=\big[M^\vee(\mu,\eta):L(\lambda,\eta)\big]. 
\end{align*}
\end{theorem}
\begin{proof}
 The result follows from Corollary  \ref{cor:highest-weight} and the proof of BGG reciprocity for category $\cO$ (and more generally for highest weight categories), see \cite[Thm.\ 3.1]{Irving} and \cite[Chap. 3]{BGGcatO}.

 \end{proof} 